\documentclass[12pt]{amsart}

\usepackage{amssymb,latexsym}

\usepackage{enumerate}
\usepackage{amsmath}

\makeatletter

\@namedef{subjclassname@2010}{

  \textup{2010} Mathematics Subject Classification}

\makeatother
\newtheorem{thm}{Theorem}[section]

\newtheorem{lem}[thm]{Lemma}
\newtheorem{pro}[thm]{Proposition}
\theoremstyle{definition}

\newtheorem{rem}[thm]{Remark}

\numberwithin{equation}{section}

\newcommand{\Lo}{\mathcal{L}}
\newcommand{\X}{\mathbb{X}}
\newcommand{\pr}{\mathbb{P}}
\newcommand{\ex}{\mathbb{E}}
\newcommand{\re}{\textup{Re}}
\newcommand{\im}{\textup{Im}}
\newcommand{\D}{\mathcal{D}_{\textup{ch}}(x)}
\newcommand{\Dg}{\widetilde{\mathcal{D}}_{\textup{ch}}(x)}
\newcommand{\F}{\mathcal{F}_{\textup{ch}}(h)}
\newcommand{\sums}{\sideset{}{^\flat}\sum}
\newcommand{\sumst}{\sideset{}{^\star}\sum}

\begin{document}

\baselineskip=17pt

\title[Class numbers in a special family of real quadratic fields]
{The distribution of class numbers in a special family of real quadratic fields}
\author{Alexander Dahl}
\author{Youness Lamzouri}

\address{Department of Mathematics and Statistics,
York University,
4700 Keele Street,
Toronto, ON,
M3J1P3
Canada}
\email{aodahl@mathstat.yorku.ca}
\email{lamzouri@mathstat.yorku.ca}

\date{}

\begin{abstract}  
We investigate the distribution of class numbers in the family of real quadratic fields $\mathbb{Q}(\sqrt{d})$ corresponding to fundamental discriminants of the form $d=4m^2+1$, which we refer to as Chowla's family. Our results show a strong similarity between the distribution of class numbers in this family and that of class numbers of imaginary quadratic fields. As an application of our results, we prove that the average order of the number of quadratic fields in Chowla's family with class number $h$ is $(\log h)/2G$, where $G$ is Catalan's constant. With minor modifications, one can obtain similar results for Yokoi's family of real quadratic fields $\mathbb{Q}(\sqrt{d})$, which correspond to fundamental discriminants of the form $d=m^2+4$.

\end{abstract}


\maketitle

\section{Introduction}

A fundamental problem in number theory is to understand the size of the class group of an algebraic number field $K$. This quantity, called the class number  of $K$, is a measure of how badly factorization in the ring of integers of $K$ fails to be unique. The case of quadratic fields had received great attention, and its rich history stretches back to the work of Gauss. Let $d$ be a fundamental discriminant and $h(d)$ be the class number of the quadratic field $\mathbb{Q}(\sqrt{d})$. Gauss conjectured that $h(d)\to \infty$ as $d\to -\infty$, and asked for the determination of all imaginary quadratic fields with a given class number $h$, a question that became known as the Gauss class number problem. The former conjecture of Gauss was proved by Heilbronn, and his class number problem for $h=1$ was solved by Heegner, Baker, and Stark. We now have a complete list of all imaginary quadratic fields with class number $h$ for all $h\leq 100$ thanks to the work of Watkins \cite{Wa}. 

Unlike imaginary quadratic fields, very little is known about real quadratic fields. In this case, Gauss conjectured that there are infinitely many real quadratic fields with class number $1$, a problem that is still open. The main difference with imaginary quadratic fields is the existence of non-trivial units in $\mathbb{Q}(\sqrt{d})$ if $d>0$, which heavily affect the size of the class number $h(d)$ in this case. Indeed, Dirichlet's class number formula states that for $d>0$ we have
\begin{equation}\label{Class1}
h(d)=\frac{\sqrt{d}}{\log\varepsilon_d}L(1,\chi_d),
\end{equation}
where $\chi_d=\left(\frac{d}{\cdot}\right)$ is the Kronecker symbol, and
$\varepsilon_d$ is the fundamental unit of the quadratic field  $\mathbb{Q}(\sqrt{d})$, defined as  
$\varepsilon_d= (a+b\sqrt{d})/2,$
where $a$ and $b$ are the smallest positive integer solutions to the Pell equations $a^2- b^2d=\pm 4$. 

Although it is a difficult problem to estimate the fundamental unit $\varepsilon_d$ in general, there exist several  families of real quadratic fields for which $\varepsilon_d$ is small in terms of $d$, and hence for which $h(d)$ is large. 
One important example is the family of real quadratic fields $\mathbb{Q}(\sqrt{d})$ corresponding to fundamental discriminants $d$ the form $4m^2+1$, where $m$ is a positive integer. This family was first studied by Chowla, who conjectured that for any positive integer $m>13$ such that $4m^2+1$ is squarefree, we have $h(4m^2+1)>1$ (see \cite{ChFr}). Another example is the family of fields $\mathbb{Q}(\sqrt{d})$ corresponding to fundamental discriminants $d$ of the form $m^2+4$, which was studied by Yokoi in \cite{Yo}. In particular, he conjectured that $h(m^2+4)>1$ for all $m>17$.  Both Chowla's and  Yokoi's conjectures were settled by Bir\'o in \cite{Bi1} and \cite{Bi2}.  There are further generalizations of Chowla's and Yokoi's families, commonly known as real quadratic fields of Richaud-Degert type. The class number problem for these fields was studied by several authors, notably by Mollin and Williams \cite{Mo}, \cite{MoWi}. 

Here and throughout we denote by $\mathcal{D}_{\textup{ch}}$ Chowla's family of fundamental discriminants, defined by
$$\mathcal{D}_{\textup{ch}}:= \{ d: d \text{ squarefree of the form } d=4m^2+1 \text{ for } m\geq 1\}.$$ 
We also let $\D=\{d\leq x: d\in \mathcal{D}_{\text{ch}}\}$. Then it follows from Lemma 1 of \cite{MoWe} that\begin{equation} \label{MontWein}
|\D|= \frac{\sqrt{x}}{2} \prod_{p>2}\left(1-\frac{c(p)}{p^2}\right) +O\left(x^{1/3}\log x\right),
\end{equation}
where
$c(p):= 1+ \left(\frac{-1}{p}\right)$. If $d\in \mathcal{D}_{\text{ch}}$ then the class number formula \eqref{Class1} becomes
\begin{equation}\label{Class2}
 h(d)=\frac{\sqrt{d}}{\log (\sqrt{d-1}+\sqrt{d})}L(1,\chi_d),
\end{equation}
since the fundamental unit is $\varepsilon_d=2m+\sqrt{d}$ if $d=4m^2+1$  is squarefree.
Therefore, assuming the generalized Riemann hypothesis GRH, we have 
\begin{equation}\label{BoundClassChowla}
\left(e^{-\gamma}\zeta(2)+o(1)\right)\frac{\sqrt{d}}{\log d \log\log d}
\leq
 h(d)
 \leq  (4e^{\gamma}+o(1))\frac{\sqrt{d}}{\log d}\log\log d,
\end{equation}
 for any $d\in \mathcal{D}_{\textup{ch}}$, where $\gamma$ is the Euler-Mascheroni constant. These bounds follow from the corresponding bounds for $L(1,\chi_d)$ obtained by Littlewood \cite{Li} under GRH. Note that the upper bound in \eqref{BoundClassChowla} holds for all real quadratic fields, since $\varepsilon_d\geq \sqrt{d}/2$ for all positive fundamental discriminants. 
 
Chowla's family $\mathcal{D}_{\textup{ch}}$ was used by Montgomery and Weinberger \cite{MoWe} to produce real quadratic fields with extremely large class numbers. 
 More precisely, they proved that there are at least $x^{3/8}$ discriminants $d\in \D$ such that  
$$
h(d)\gg  \frac{\sqrt{d}}{\log d}\log\log d.
$$
This result was recently refined by Lamzouri \cite{La3}, who showed that there are at least $x^{1/2-1/\log\log x}$ discriminants $d\in \D$ such that
\begin{equation}\label{Lamzouri15}
h(d)\geq (2e^{\gamma}+o(1))\frac{\sqrt{d}}{\log d}\log\log d.
\end{equation}
The lower bound \eqref{Lamzouri15} is believed to be best possible over all positive fundamental discriminants $d$, in view of the widely believed conjecture that $L(1,\chi_d)\leq (e^{\gamma}+o(1))\log\log|d|$ for all fundamental discriminants $d$.  Note that the true lower bound for $L(1,\chi_d)$ is believed to be $\big(e^{-\gamma}\zeta(2)+o(1)\big)/\log\log |d|$, which would imply a lower bound for $h(d)$ over $d\in\mathcal{D}_{\text{ch}}$ that is twice as large as the GRH lower bound in \eqref{BoundClassChowla}. One can refer to \cite{GrSo2} for a discussion and results related to these conjectures. 

In this paper, we shall investigate the distribution of $h(d)$ over fundamental discriminants $d$ in Chowla's family.  With minor modifications, one can obtain similar results for Yokoi's family of real quadratic fields. Here and throughout we let $\log_j$ be the $j$-fold iterated logarithm; that is, $\log_2=\log\log$, $\log_3=\log\log\log$ and so on. Our main result  shows that the tail of the distribution of large (and small) values of $h(d)$  over $d\in \mathcal{D}_{\text{ch}}$ is double exponentially decreasing. In particular, it implies \eqref{Lamzouri15}.
\begin{thm}\label{MainResult} Let $x$ be large, and $1\leq \tau\leq \log_2 x-3\log_3x$. The number of discriminants $d\in \D$ such that 
$$ h(d)\geq 2e^{\gamma}\frac{\sqrt{d}}{\log d} \cdot \tau,$$
equals 
\begin{equation}\label{TailDistributionClass} |\D| \cdot  \exp\left(-\frac{e^{\tau-C_0}}{\tau}\left(1+ O\left(\frac{1}{\tau}\right)\right)\right),
\end{equation}
where 
\begin{equation}\label{SpecialConstant}
C_0:= \int_0^1\frac{\tanh(t)}{t}dt + \int_1^{\infty}\frac{\tanh(t)-1}{t}dt=0.8187\cdots.
\end{equation}
 Moreover, the same estimate holds for the number of discriminants $d\in \D$ such that 
$$ h(d)\leq 2e^{-\gamma}\zeta(2)\frac{\sqrt{d}}{\log d}\cdot \frac{1}{\tau},$$
in the same range of $\tau$. 
\end{thm}

In view of the class number formula \eqref{Class2}, the distribution of $h(d)$ is completely determined by that of $L(1,\chi_d)$ over $d\in \mathcal{D}_{\text{ch}}$.   Our strategy is to compare the distribution of $L(1,\chi_d)$ over $d\in \mathcal{D}_{\text{ch}}$ to that of a random Euler product 
$L(1,\X)=\prod_{p}\left(1-\X(p)/p\right)^{-1},$
where the $\X(p)$'s are independent random variables  taking the values $0, \pm 1$ with suitable probabilities that are described below. One can think of the random variable $\X(p)$ as a model for the value of $\chi_d(p)$ as $d$ varies in  $\mathcal{D}_{\text{ch}}$. 

One should compare our results with those of Granville and Soundararajan \cite{GrSo2} concerning the distribution of values of $L(1,\chi_d)$ over all fundamental discriminants $d$ such that $|d|\leq x$ (their results also hold if one restricts attention to either positive or negative discriminants). Although the probabilistic random model for this family is different from that of Chowla's family (for arithmetic reasons that are explained below), the tail of the distribution of these values satisfy a similar estimate to \eqref{TailDistributionClass}. In particular, one can deduce from their results that the proportion of imaginary quadratic fields $\mathbb{Q}(\sqrt{-d})$ with $d\leq x$ such that $h(-d)\geq \frac{\sqrt{d}}{\pi} e^{\gamma}\tau$ (or $h(-d)\leq \frac{\zeta(2)\sqrt{d}}{\pi} (e^{\gamma}\tau)^{-1}$) equals $\exp\left(-\frac{e^{\tau-C_0}}{\tau}\left(1+ O\left(\frac{1}{\tau}\right)\right)\right)$,  in asymptotically the same range for $\tau$. This shows a strong similarity between the distribution of class numbers in Chowla's family and that of class numbers of imaginary quadratic fields.

Let $\{\X(p)\}_{p}$ be a sequence of independent random variables taking the value $1$ with probability $\alpha_p$, $-1$ with probability $\beta_p$, and  $0$ with probability $\gamma_p$, where $\alpha_2 = \beta_2 = 1/2$, $\gamma_2=0$, and for odd $p$ we have
\begin{equation}\label{ALPHAP}
\alpha_p= \frac{1}{2} \left(1-\frac{c(p)+1}{p}\right)\left(1-\frac{c(p)}{p^2}\right)^{-1},
\end{equation}
\begin{equation}\label{BETAP}
\beta_p= \frac{1}{2} \left(1-\frac{c(p)-1}{p}\right)\left(1-\frac{c(p)}{p^2}\right)^{-1},
\end{equation}
and
$$ \gamma_p=\frac{p c(p)-c(p)}{p^2-c(p)}=1-\left(1-\frac{c(p)}{p}\right)\left(1-\frac{c(p)}{p^2}\right)^{-1}.$$
The argument for choosing these probabilities is as follows: Let $p$ be an odd prime. If $d=4m^2+1$ is squarefree then $d$ lies in one of $p^2-c(p)$ residue classes modulo $p^2$, since $p^2\nmid 4m^2+1$.  Among these, $\chi_d(p)=0$ for exactly $pc(p)-c(p)$ of them, which justifies the choice of $\gamma_p$. Furthermore, since $d$  belongs to one of $p^2-c(p)$ residue classes modulo $p^2$, then we must have 
\begin{equation}\label{Jacobsthal}
\alpha_p- \beta_p=\ex(\X(p))= \left(\frac{p^2-c(p)}{p^2}\right)^{-1}\left(\frac1p\sum_{m=0}^{p-1}\left(\frac{4m^2+1}{p}\right)\right)=-\frac{1}{p}\left(1-\frac{c(p)}{p^2}\right)^{-1}
\end{equation}
which follows from the Jacobsthal sum identity $\sum_{m=0}^{p-1}\left(\frac{4m^2+1}{p}\right)=-1$ (see for example \cite{St}). Combining \eqref{Jacobsthal} with the fact that $\alpha_p+\beta_p=1-\gamma_p$ yield \eqref{ALPHAP} and \eqref{BETAP}.

\noindent For the prime $2$, note that $4m^2+1$ lies in one of the residue classes $1, 5 \pmod 8$, and the values $\pm1$ occur equally often. 

We extend the $\X(p)$'s multiplicatively to all positive integers by setting $\X(1)=1$ and 
$ \X(n):= \X(p_1)^{a_1}\cdots \X(p_k)^{a_k} $ if $n= p_1^{a_1}\cdots p_k^{a_k}.$
We now define
$$L(1,\X):=\sum_{n=1}^{\infty}\frac{\X(n)}{n}= \prod_{p}\left(1-\frac{\X(p)}{p}\right)^{-1},$$
where both the series and the product are almost surely convergent by Lemma \ref{ExpectationRandom} below together with Kolmogorov's three-series Theorem. 
For $\tau>0$, define 
$$ \Phi_{\X}(\tau):= \pr\big(L(1,\X)>e^{\gamma}\tau\big) \text{ and } \Psi_{\X}(\tau):= \pr\left(L(1,\X)<\frac{\zeta(2)}{e^{\gamma}\tau}\right).
$$
We prove that the distribution of $L(1,\chi_d)$ over $d\in \mathcal{D}_{\text{ch}}$ is very well approximated by that of $L(1,\X)$ uniformly in almost all of the viable range. 
\begin{thm}\label{TheoremDistribution}
Let $x$ be large. Uniformly in the range $1\leq \tau\leq \log_2 x-2\log_3x-\log_4 x$, we have 
$$\frac{1}{|\D|}\big|\{d\in \D: L(1,\chi_d) >e^{\gamma}\tau \}\big|= \Phi_{\X}(\tau)\left(1+O\left(\frac{e^{\tau}(\log_2 x)^2\log _3 x}{\log x}\right)\right),$$
and 
$$ \frac{1}{|\D|}\left|\left\{d\in \D: L(1,\chi_d) <\frac{\zeta(2)}{e^{\gamma}\tau} \right\}\right|= \Psi_{\X}(\tau)\left(1+O\left(\frac{e^{\tau}(\log_2 x)^2\log _3 x}{\log x}\right)\right).$$
\end{thm} 
In order to deduce Theorem \ref{MainResult}, we need to study the asymptotic behaviour of the distribution functions $\Phi_{\X}(\tau)$ and $\Psi_{\X}(\tau)$  in terms of $\tau$ when $\tau$ is large. We accomplish this by a careful saddle point analysis. 
\begin{thm}\label{ExponentialDecay}
For large $\tau$ we have  
\begin{equation}\label{AsympLargeDeviations}
\Phi_{\X}(\tau)=\exp\left(-\frac{e^{\tau-C_0}}{\tau}\left(1+ O\left(\frac{1}{\tau}\right)\right)\right),
\end{equation}
where $C_0$ is defined in \eqref{SpecialConstant}. 
The same estimate also holds for $\Psi_{\X}(\tau)$. Moreover, if $0\leq \lambda\leq e^{-\tau}$, then we have 
\begin{equation}\label{Perturbations}
\Phi_{\X}\left(e^{-\lambda} \tau\right)=\Phi_{\X}(\tau) \big(1+O\left(\lambda e^{\tau}\right)\big), \textup{ and } \Psi_{\X}\left(e^{-\lambda} \tau\right)=\Psi_{\X}(\tau) \big(1+O\left(\lambda e^{\tau}\right)\big).
\end{equation}

\end{thm}

Our proof of Theorem \ref{TheoremDistribution} relies on computing complex moments of $L(1,\chi_d)$ over $d\in \mathcal{D}_{\text{ch}}$. To this end, we show that the average of $L(1,\chi_d)^z$ over Chowla's family is asymptotically equal to the corresponding moments of the probabilistic random model $L(1,\X)$ uniformly in a wide range of the complex variable $z$. 
\begin{thm}\label{ComplexMoments}
Let $x$ be large. There exists a positive constant $B$ such that uniformly for all complex numbers $z$ with $|z|\leq B\log x/(\log_2x\log_3 x)$ we have
$$\frac{1}{|\D|}\sumst_{d\in \D} L(1,\chi_d)^z= \ex\left(L(1,\X)^z\right)+
O\left(\exp\left(-\frac{\log x}{20\log_2x}\right)\right),
$$
where $\sumst$ indicates that the sum is over non-exceptional discriminants $d$.
\end{thm}

\begin{rem}\label{RemExceptional} The precise definition of an exceptional discriminant $d$ is stated in \eqref{Exceptional}. Note that if $d$ is exceptional, we could have $
L(1,\chi_d)$ as small as $d^{-\epsilon}$, so that when $|z|$ is large and $z<0$, the $z$-th moment of $L(1,\chi_d)$ would be heavily affected by the contribution of this particular discriminant. This justifies  the assumption that $d$ is non-exceptional in Theorem \ref{ComplexMoments}. Furthermore, note that if $\text{Re}(z)<0$ but $|\text{Re}(z)|$ is bounded, we no longer need the condition that $d$ is non-exceptional in Theorem \ref{ComplexMoments}, thanks to Siegel's bound $L(1,\chi_d)\gg_{\epsilon} d^{-\epsilon}$.
\end{rem}

As an application of our results, we investigate the number of discriminants in the family $\mathcal{D}_{\text{ch}}$ with class number $h$, which we denote by $\F$. The number of imaginary quadratic fields with class number $h$ was studied by Soundararajan in \cite{So}. In particular, he developed an asymptotic formula for its average value, a result whose error term was improved upon in \cite{La4}. A variant of Soundararajan's asymptotic formula (over odd $h$) was recently used by Holmin, Jones, Kurlberg, McLeman and Petersen \cite{HJKMP} to investigate statistics of class numbers of imaginary quadratic fields. 

By the class number formula \eqref{Class2}, one expects that the main contribution to the average of $\F$ over $h\leq H$ comes from discriminants $d$ of size $\ll H^2(\log H)^2$, since $L(1,\chi_d)$ is constant on average (by Theorem \ref{ComplexMoments}). Since there are $\asymp H\log H$ such discriminants in $\mathcal{D}_{\text{ch}}$, this heuristic argument suggests that the  average size of $\F$ should be around $\log h$. We prove that this is indeed the case. 

\begin{thm}
\label{thm:F-average}
As $H \rightarrow \infty$, we have
\[
	\sum_{h\leq H} \F = \frac{1}{2G} H \log H 
	+ O\left(H(\log_2 H)^2\log_3 H\right),
\]
where $$G=L(2,\chi_{-4})=1-\frac{1}{3^2}+\frac{1}{5^2}-\frac{1}{7^2}+\frac{1}{9^2}+\cdots=0.916...$$ is Catalan's constant, and $\chi_{-4}$ is the non-principal character modulo 4.
\end{thm}

This paper is organized as follows: In Section \ref{Sec2} we establish an asymptotic formula for the average value of $\chi_d(m)$ over $d$ in $\D$. In particular, we show that in a certain range of $m$ in terms of $x$, the average order of $\chi_d(m)$ equals $\ex(\X(m))$. This is used to compute  complex moments of $L(1,\chi_d)$ over $d\in \D$ and prove Theorem \ref{ComplexMoments} in Section \ref{Sec3}. In Section \ref{Sec4} we use the saddle-point method to study the distribution of the random Euler product $L(1, \X)$  and prove Theorem \ref{ExponentialDecay}. These results are then used to prove Theorems \ref{MainResult}  and \ref{TheoremDistribution}  in Section \ref{Sec5}. Finally, we apply our results to study $\F$ and prove Theorem \ref{thm:F-average} in Section \ref{Sec6}.

%%%%%%%%%%%%%%%%%%%%%%%%%%%%%%%%%%%%%%%%%%%%%%%%%%
\section{An asymptotic formula for the character sum $\sum_{d\in \D} \chi_d(m)$}\label{Sec2}
In order to prove that the moments of $L(1,\chi_d)$ over $d\in \D$ are nearly equal to the corresponding moments of $L(1,\X)$, we first need to show that the average order of $\chi_d(m)$ equals $\ex(\X(m))$ when $m$ is small compared to $x$. 
\begin{pro} \label{AsympChar}
Let $m$ be a positive integer. Then we have
$$\frac{1}{|\D|}\sum_{d\in \D} \chi_d(m)= \ex(\X(m))+O\left(m^{2/3}x^{-1/6}\log x\right).
$$
\end{pro}
To prove this result, we first need the following lemmas. Here and throughout we let $\omega(n)$ be the number of distinct prime factors of $n$.
\begin{lem}\label{ExpectationRandom} 
Let $m=2^\ell p_1^{a_1}\cdots p_k^{a_k}$ be the prime factorization of $m$, and let $m_0$ be the squarefree part of $p_1^{a_1} \cdots p_k^{a_k}$. Then we have
$$
\ex(\X(m))=
\begin{cases}
\displaystyle
	\frac{1}{m_0}(-1)^{\omega(m_0)}\prod_{\substack{1\leq j\leq k\\2\mid a_j}}
	\left(1-\frac{c(p_j)}{p_j}\right)\prod_{j=1}^k\left(1-\frac{c(p_j)}{p_j^2}\right)^{-1}
	& \text{ if } \ell \text{ is even},\\
	0 & \text{ if } \ell \text{ is odd}.
\end{cases}
$$
\end{lem}
\begin{proof}
Using the independence of the $\X(p)$'s we obtain
\begin{equation}\label{Multiplicative}
\ex(\X(m))= \ex\left(\X(2)^\ell\right) \prod_{j=1}^k\ex\left(\X(p_j)^{a_j}\right).
\end{equation}
First, if $a_j$ is even then
$$ \ex\left(\X(p_j)^{a_j}\right)= \alpha_{p_j}+\beta_{p_j}=1-\gamma_{p_j}
=
\left(1-\frac{c(p_j)}{p_j}\right)\left(1-\frac{c(p_j)}{p_j^2}\right)^{-1}.
$$
On the other hand, if $a_j$ is odd then
$$\ex\left(\X(p_j)^{a_j}\right)=\alpha_{p_j}-\beta_{p_j}=-\frac{1}{p_j}\left(1-\frac{c(p_j)}{p_j^2}\right)^{-1}.
$$
Finally, note that $\ex\left(\X(2)^\ell\right)$ equals $1$ is $\ell$ is even, and $0$ otherwise. 
Inserting these estimates in \eqref{Multiplicative} completes the proof.
\end{proof}

\begin{lem}\label{ExactChar}
Let $m=2^\ell p_1^{a_1}\cdots p_k^{a_k}$ be the prime factorization of $m$, and let $m_0$ be the squarefree part of $p_1^{a_1}\cdots p_k^{a_k}$. Then we have
$$
\frac{1}{m}\sum_{n=1}^{m}\left(\frac{4n^{2}+1}{m}\right)=
\begin{cases}
\displaystyle
	\frac{1}{m_0}(-1)^{\omega(m_0)}\prod_{\substack{1\leq j\leq k\\2\mid a_j}}
	\left(1-\frac{c(p_j)}{p_j}\right)
	& \text{ if } \ell \text{ is even},\\
	0 & \text{ if } \ell \text{ is odd}.
\end{cases}
$$
\end{lem}

\begin{proof}  Let $g(n)=4n^{2}+1$. Observe that the sum $\sum_{n=1}^m\left(g(n)/m\right)$ is a complete character sum, and hence by multiplicativity and the Chinese remainder theorem, we have 
\[
\sum_{n=1}^{m}\left(\frac{g(n)}{m}\right)
=\sum_{n_0=1}^{2^\ell}\left(\frac{g(n_0)}{2}\right)^{\ell}
\prod_{j=1}^{k}\left(\sum_{n_{j}=1}^{p_{j}^{a_j}}\left(\frac{g(n_j)}{p_{j}}\right)^{a_j}\right).
\]
If $a_j$ is even then
\[
\sum_{n_j=1}^{p_{j}^{a_j}}\left(\frac{g(n_j)}{p_j}\right)^{a_j}=p_j^{a_j-1}(p_{j}-c(p_{j}))
=p^{a_j}\left( 1-\frac{c(p_j)}{p_j} \right),
\]
since there are exactly $c(p_j) p_j^{a_j-1}$ integers $n_j$ such that $1\leq n_j\leq p_j^{a_j}$ and $g(n_j)\equiv 0\pmod{p_j}$. 
On the other hand, if $a_j=2b_j+1$ is odd then
\[
\sum_{n_{j}=1}^{p_{j}^{a_{j}}}\left(\frac{g(n_{j})}{p_{j}}\right)^{a_{j}}=\sum_{n_{j}=1}^{p_{j}^{a_{j}-1}}\sum_{c=1}^{p_{j}}\left(\frac{g(n_{j}p_{j}+c)}{p_{j}}\right)^{2b_{j}+1}=p_{j}^{a_{j}-1}\sum_{c=1}^{p_{j}}\left(\frac{g(c)}{p_{j}}\right)= -p_{j}^{a_{j}-1},
\]
since $\sum_{m=0}^{p-1}\left(\frac{4m^2+1}{p}\right)=-1$.

Finally, note that $\left(\frac{g(b_0)}{2}\right)= \left(\frac{4b_0^2+1}{2}\right)$ equals $1$ if $b_0$ is even, and $-1$ otherwise. Hence, it follows that 
\[
\sum_{b_0=1}^{2^\ell}\left(\frac{g(b_0)}{2}\right)^{\ell}
=
\begin{cases}
	2^\ell, & \ell \text{ even}, \\
	0, & \ell \text{ odd}.
\end{cases}
\]
Combining the above estimates completes the proof.

\end{proof}

\begin{proof}[Proof of of Proposition \ref{AsympChar}]
To simplify our notation, we define $S(x)=\sum_{d\in \D} \chi_d(m)$, and put 
 $y= \sqrt{x-1}/2$. Then, using that $\mu^2(n)=\sum_{r^2\mid n} \mu(r)$ we obtain
\begin{align*}
S(x)&=\sum_{n\leq y}\left(\frac{4n^{2}+1}{m}\right)\mu^{2}(4n^{2}+1)
= \sum_{n\leq y}\left(\frac{4n^{2}+1}{m}\right)\sum_{r^2\mid 4n^2+1}\mu(r)\\
&= \sum_{\substack{r\leq \sqrt{x}\\
(r, 2m)=1}}\mu(r)\sum_{\substack{n\leq y\\
r^{2}\mid4n^{2}+1}}\left(\frac{4n^{2}+1}{m}\right).
\end{align*}
Let $2\leq T\leq y$ be a real parameter to be chosen later. We split the above sum over $r$ into two parts $r\leq T$ and $T< r\leq \sqrt{x}$. Writing $4n^2+1=r^2 s$, it follows that the contribution of the second part is 
$$ \ll \sum_{T< r\leq \sqrt{x}} \sum_{\substack{n\leq y\\
r^{2}\mid4n^{2}+1}} 1\ll \sum_{s\leq x/T^2}\sum_{\substack{n, r\\
(2n)^2-sr^2=-1}} 1.$$
From the theory of Pell's equation, the number of pairs $(u, v)$ for which $1\leq u\leq U$ and $u^2-sv^2=-1$, is $\ll \log U$ uniformly in $s$. Hence, we deduce that the contribution of the terms $T\leq r\leq \sqrt{x}$ to $S(x)$ is $\ll x(\log x)/T^2.$ Thus, 
\begin{equation}\label{CharSumEsti}
S(x)=\sum_{\substack{r\leq T\\
(r, 2m)=1
}
}\mu(r)\sum_{\substack{n\leq y\\
r^{2}\mid4n^{2}+1
}
}\left(\frac{4n^{2}+1}{m}\right)+O\left(\frac{x\log x}{T^{2}}\right).
\end{equation}

Let $r\leq T$ such that $(r, 2m)=1$, and consider the equation
$
4n^{2}+1\equiv0\pmod{r^{2}}. 
$
This congruence has $c(r^{2})=c(r)$ solutions  modulo $r^{2}$ where
$
c(r)=\prod_{p\mid r} c(p).
$
Denote these solutions by $\{a_{1},...,a_{c(r)}\}$. Then, for any integer $k$ we have
\begin{align*} 
\sum_{\substack{k r^2 m < n\leq (k+1)r^2 m\\
r^{2}\mid4n^{2}+1}}\left(\frac{4n^{2}+1}{m}\right)
&=\sum_{i=1}^{c(r)}\sum_{\substack{k r^2 m < n\leq (k+1)r^2 m\\
n\equiv a_i \text{ mod }r^{2}}}\left(\frac{4n^{2}+1}{m}\right)\\
&=\sum_{i=1}^{c(r)}\sum_{u=1}^m \left(\frac{4u^{2}+1}{m}\right)\sum_{\substack{k r^2 m< n\leq (k+1)r^2 m\\
n\equiv a_i \text{ mod }r^{2}\\ n\equiv u\text{ mod }m}} 1\\
&= c(r) \sum_{u=1}^m\left(\frac{4u^{2}+1}{m}\right),
\end{align*}
by the Chinese remainder theorem, since $(r, m)=1$. Therefore, we deduce that
\begin{align*}
\sum_{\substack{n\leq y\\
r^{2}\mid4n^{2}+1
}
}\left(\frac{4n^{2}+1}{m}\right) 
&=y\frac{c(r)}{r^2}\frac{1}{m}\sum_{u=1}^m\left(\frac{4u^{2}+1}{m}\right)+ O\big(c(r)m\big)\\
&= y \frac{c(r)}{r^2} \ex(\X(m))\prod_{\substack{p \mid m\\ p>2}}\left(1-\frac{c(p)}{p^2}\right)+ O\big(c(r)m\big),\\
\end{align*}
by Lemmas \ref{ExpectationRandom} and \ref{ExactChar}. Inserting this estimate in \eqref{CharSumEsti} we get
$$ 
S(x)= y \cdot \ex(\X(m))\prod_{\substack{p \mid m\\ p>2}}\left(1-\frac{c(p)}{p^2}\right) \sum_{\substack{r\leq T\\
(r, 2m)=1}}\mu(r)\frac{c(r)}{r^2} + O\left(m\sum_{r\leq T}c(r)+\frac{x\log x}{T^2}\right).
$$
Since $c(r)\leq 2^{\omega(r)}\leq  d(r)$ (where $d(r)$ is the divisor function), we get $\sum_{r\leq T} c(r)\ll T\log T$ and 
\[
\sum_{\substack{r> T\\
(r, 2m)=1}}\frac{\mu(r)}{r^{2}}c(r)\ll \sum_{r>T} \frac{d(r)}{r^2}\ll \frac{\log T}{T},
\]
by using that $\sum_{r\leq t} d(r)\sim t\log t$, together with partial summation.  Thus, we deduce
\begin{align*}
S(x)&= y \cdot \ex(\X(m))\prod_{\substack{p \mid m\\ p>2}}\left(1-\frac{c(p)}{p^2}\right) \sum_{\substack{r\geq 1\\
(r, 2m)=1}}\mu(r)\frac{c(r)}{r^2} + O\left(mT\log T+\frac{\sqrt{x}\log T}{T}+\frac{x\log x}{T^2}\right) \\
&= y \cdot \ex(\X(m))\prod_{p>2}\left(1-\frac{c(p)}{p^2}\right)+ O\left(mT\log T+\frac{\sqrt{x}\log T}{T}+\frac{x\log x}{T^2}\right).
\end{align*}
Choosing
$
T=\big(x/m\big)^{1/3}
$
and using \eqref{MontWein} completes the proof. 
\end{proof}

%%%%%%%%%%%%%%%%%%%%%%%%%%%%%%%%%%%%%%%%%%%%%
\section{Complex moments of $L(1,\chi_d)$ over $d\in \D$: Proof of Theorem \ref{ComplexMoments}}\label{Sec3}

For any $z\in \mathbb{C}$, we have 
$$L(1, \X)^z =\sum_{n=1}^{\infty} \frac{d_z(n)}{n}\X(n) $$
almost surely, where $d_z(n)$ is the $z$-th divisor function. Recall that $d_z(n)$ is the multiplicative function defined on prime powers by $d_z(p^a)=\Gamma(z+a)/(\Gamma(z)a!)$, and  for $\text{Re}(s)>1$ we have 
 $$\sum_{n=1}^{\infty} \frac{d_z(n)}{n^s}=\zeta(s)^z.$$
  We observe that
$|d_z(n)|\leq d_{|z|}(n)\leq d_k(n)$
for any integer $k\geq |z|$, and $d_k(mn)\leq d_k(m)d_k(n)$ for any positive integers $k,m,n$.  Furthermore
for $k\in {\Bbb N}$, and $y>3$ we have that
 $$
 d_k(n)e^{-n/y}\leq
e^{k/y}\sum_{a_1...a_k=n}e^{-(a_1+...+a_k)/y},
$$ and so
\begin{equation}\label{Divisor2}
 \sum_{n=1}^{\infty}\frac{d_k(n)}{n}e^{-n/y}\leq \left(e^{1/y}
\sum_{a=1}^{\infty}\frac{e^{-a/y}}{a}\right)^k\leq (\log
3y)^k.
\end{equation}

In order to prove Theorem \ref{ComplexMoments}, we first need some preliminary results. 
We define a discriminant $d$ to be {\it exceptional} if there exists a complex number $s$ such that  $L(s,\chi_d)=0$ and \begin{equation}\label{Exceptional}
\text{Re}(s)\geq 1-\frac{c}{\log (|d|(\text{Im}(s)+2))}
\end{equation} for some sufficiently small constant $c>0$. 
One expects that there are no such discriminants, but what is known unconditionally is that these discriminants, if they exist, must be very rare.
Indeed, it is shown in Chapter 14 of \cite{Da} that between any
two powers of $2$ there is at most one exceptional discriminant $d$. In particular, it follows that there are at most $O(\log x)$ such discriminants up to $x$. 

 If $\chi$ is a non-principal and non-exceptional Dirichlet character modulo $q$, then we have the following standard bound for $\log L(1+it, \chi)$ (see for example Lemma 2.2 of \cite{La1})
\begin{equation}\label{StandardBound}
\log L(1+it, \chi)\ll \log_2 \big(q(|t|+2)\big).
\end{equation}
We can obtain a much better bound for $\log L(s,\chi)$, with $s$ close to $1$, if $L(s,\chi)$ has no zeros in a certain  rectangle containing $s$.
\begin{lem}\label{BoundLZeroFree} Let $q$ be large and put $\eta=1/\log_2q$. Let $0<\epsilon<1/2$ be fixed.  Assume that $L(z,\chi)$ has no zeros in the rectangle $\{z: 1-\epsilon \leq \text{Re}(z)\leq 1\text{ and } |\text{Im}(z)|\leq 2(\log q)^{2/\epsilon}\}.$ Then for any $s=\sigma+it$ with $1-\eta\leq \sigma\leq 1$ and $|t|\leq \log^4q$ we have
$$ |\log L(s,\chi)|\leq \log_3 q+O_{\epsilon}(1).$$
\end{lem}
To prove this result we need the following lemma from \cite{GrSo1}.
\begin{lem}[Lemma 8.2 of \cite{GrSo1}] \label{lemGrSo}
Let $s=\sigma+it$ with $\sigma>1/2$ and $|t|\leq 2q$. Let $y\geq 2$ be a real number, and let $1/2\leq \sigma_0<\sigma$. Suppose that the rectangle $\{z:\sigma_0<\text{Re}(z)\leq 1, |\text{Im}(z)-t|\leq y+3\}$ contains no zeros of $L(z,\chi)$. Put  $\sigma_1=\min(\frac{\sigma+\sigma_0}{2},\sigma_0+\frac{1}{\log y})$. Then
$$ \log L(s,\chi)=\sum_{n=2}^y\frac{\Lambda(n)\chi(n)}{n^s\log n}+O\left(\frac{\log q}{(\sigma_1-\sigma_0)^2}y^{\sigma_1-\sigma}\right).$$
\end{lem}

\begin{proof}[Proof of Lemma \ref{BoundLZeroFree}]
We use Lemma \ref{lemGrSo} with $1-\eta\leq \sigma\leq 1$, $\sigma_0=1-\epsilon$ and $y=(\log q)^{2/\epsilon}$.  Therefore,  if $L(z,\chi)$ has no zeros in the rectangle $\{z: 1-\epsilon \leq \text{Re}(z)\leq 1\text{ and } |\text{Im}(z)|\leq 2(\log q)^{2/\epsilon}\},$ we get
$$
|\log L(s,\chi)|=\left|\sum_{p\leq (\log q)^{2/\epsilon}}\frac{\chi(p)}{p^s}\right|+O(1)\leq \sum_{p\leq(\log q)^{2/\epsilon}}\frac{1+O(\eta \log p)}{p}+O(1)\leq \log_3 q+O_{\epsilon}(1).
$$
\end{proof}

Using Lemma \ref{BoundLZeroFree} we obtain the following approximation to $L(1,\chi)^z$, if $L(z,\chi)$ has no zeros in a small region to the left of the line $\re(s)=1$. 

\begin{pro}\label{ShortApproxL}
Let $q$ be large and $0<\epsilon<1/2$ be fixed. Let $y$ be a real number such that $\log q/\log_2 q\leq \log y\leq \log q$. Furthermore, assume that $L(s,\chi)$ has no zeros inside the rectangle $\{s:1-\epsilon <\text{Re}(s)\leq 1 \text{ and } |\text{Im}(s)|\leq 2(\log q)^{2/\epsilon}\}$.
Then for any complex number $z$ such that $|z|\leq  \log y/(4\log_2 q \log_3 q)$ we have
$$L(1,\chi)^z=\sum_{n=1}^{\infty}\frac{d_z(n)\chi(n)}{n}e^{-n/y}+O_{\epsilon}\left(\exp\left(-\frac{\log y}{2\log_2 q}\right)\right).$$
\end{pro}

\begin{proof} Since $\frac{1}{2\pi i}\int_{2-i\infty}^{2+i\infty}y^s\Gamma(s)ds= e^{-1/y}$ then
$$\frac{1}{2\pi i}\int_{2-i\infty}^{2+i\infty}L(1+s,\chi)^z\Gamma(s)y^sds= \sum_{n=1}^{\infty}\frac{d_z(n)\chi(n)}{n}e^{-n/y}.$$
we shift the contour to $\mathcal{C}$, where $\mathcal{C}$ is the path which joins
$$ -i\infty, -i(\log q)^4, -\eta-i(\log q)^4, -\eta+i(\log q)^4, -i(\log q)^4, +i\infty,$$
where $\eta=1/\log_2 q$.
We encounter a simple pole at $s=0$ which leaves the residue $L(1,\chi)^z$. Using the bound \eqref{StandardBound} together  with Stirling's formula we obtain
$$
\frac{1}{2\pi i}\left(\int_{-i\infty}^{-i(\log q)^4}+\int_{i(\log q)^4}^{+i\infty}\right)L(1+s,\chi)^z\Gamma(s) y^s ds\ll \int_{(\log q)^4}^{\infty}e^{O(|z|\log_2 qt)}e^{-\frac{\pi}{3}t}dt\ll \frac{1}{q}.
$$
Finally, using that $\Gamma(s)$ has a simple pole at $s=0$ together with Lemma \ref{BoundLZeroFree} and Stirling's formula, we deduce that
\begin{align*}
&\frac{1}{2\pi i}\left(\int_{-i(\log q)^4}^{-\eta-i(\log q)^4}+\int_{-\eta-i(\log q)^4}^{-\eta+i(\log q)^4}+\int_{-\eta+i(\log q)^4}^{i(\log q)^4}\right)L(1+s,\chi)^z\Gamma(s)y^s ds\\
&\ll \exp\left(-\frac{\pi}{3}(\log q)^4+O(|z|\log_3 q)\right)+ \frac{y^{-\eta}}{\eta}\exp\big(|z|\log_3q+O_{\epsilon}(|z|)
\big)(\log q)^4\\
&\ll_{\epsilon} \exp\left(-\frac{\log y}{2\log_2 q}\right).
\end{align*}

\end{proof}

We are now ready to prove Theorem \ref{ComplexMoments}.
\begin{proof}[Proof of Theorem \ref{ComplexMoments}]
Let $\Dg$ be the set of fundamental discriminants $d\in \D$ such that $d>\sqrt{x}$ and $L(s, \chi_d)$ has no zeros in the rectangle $\{s: 9/10 <\text{Re}(s)\leq 1 \text{ and } |\text{Im}(s)|\leq 2(\log x)^{20}\}$. To bound $|\D\setminus\Dg|$ we use the following zero-density result of Heath-Brown \cite{HB}, which states that for $1/2<\sigma<1$ and
 any $\epsilon>0$ we have
$$\sums_{|d|\leq x} N(\sigma,T, \chi_d)\ll (xT)^{\epsilon}x^{3(1-\sigma)/(2-\sigma)}T^{(3-2\sigma)/(2-\sigma)},$$
where $N(\sigma, T, \chi_d)$ is the number of zeros $\rho$ of $L(s,\chi_d)$ with $\re(\rho)\geq \sigma$ and $|\im(\rho)|\leq T$, and $\sums$ indicates that the sum is over fundamental discriminants. Then, it follows from this bound that
$$ |\D|-|\Dg|\ll x^{1/3}.$$
Using this estimate together with the bound \eqref{StandardBound} we obtain
\begin{equation}\label{GoodL}
\sumst_{d\in \D} L(1,\chi_d)^z- \sum_{d\in \Dg} L(1,\chi_d)^z
\ll x^{1/3} \exp\big(O(|z|\log_2 x)\big)\ll x^{3/8}.
\end{equation}
Let $y= x^{1/6} $, and put $k=\lceil |z|\rceil$.  Then, it follows from Proposition \ref{ShortApproxL} that
\begin{equation}\label{ApproxGood}
\sum_{d\in \Dg} L(1,\chi_d)^z= \sum_{d\in \Dg}
\sum_{m=1}^{\infty}\frac{d_{z}(m)\chi_{d}(m)e^{-m/y}}{m}+ O\left(|\D| \exp\left(-\frac{\log x}{20 \log\log x}\right)\right).
\end{equation} 
We now extend the main term of the last estimate, so as to include all fundamental discriminants $d\in \D$. Using  \eqref{Divisor2}, we deduce that 
$$ 
\sum_{d\in \D \setminus \Dg}
\sum_{m=1}^{\infty}\frac{d_{z}(m)\chi_{d}(m)e^{-m/y}}{m}\ll (|\D|-|\Dg|)\sum_{m=1}^{\infty}\frac{d_k(m)}{m}e^{-m/y} \ll x^{3/8}.
$$
Combining this estimate with \eqref{GoodL} and \eqref{ApproxGood}  gives
$$ 
\sumst_{d\in \D} L(1,\chi_d)^z= 
\sum_{m=1}^{\infty}\frac{d_{z}(m)}{m}e^{-m/y}\sum_{d\in \D}\chi_{d}(m)+O\left(x^{1/2}\exp\left(-\frac{\log x}{20 \log\log x}\right)\right).
$$

Now, it follows from Proposition \ref{AsympChar} that
\begin{multline}
\frac{1}{|\D|}\sum_{m=1}^{\infty}\frac{d_{z}(m)}{m}e^{-m/y}\sum_{d\in \D}\chi_{d}(m)=\sum_{m=1}^{\infty}\frac{d_{z}(m)\mathbb{E}(X(m))}{m}e^{-m/y}
\\
+O\left(x^{-1/6}\log x\sum_{m=1}^{\infty}\frac{d_k(m)}{m^{1/3}}e^{-m/y}\right).\label{eq:expected-error-1}
\end{multline}
To bound the error term in the last estimate, we split the sum into two parts: $m\leq y \log^2 y$ and $m> y \log^2 y$. The contribution of the first part is 
$$
 \leq  \sum_{m\leq y\log^2 y} \left( \frac{y \log^2 y}{m} \right)^{2/3} \frac{d_k(m)}{m^{1/3}} e^{-m/y}
 \leq (y\log^2 y)^{2/3}\sum_{m=1}^{\infty} \frac{d_k(m)}{m} e^{-m/y}
\ll y^{2/3} (\log 3y)^{k+4/3},
$$
by \eqref{Divisor2}.
The remaining terms contribute
\begin{align*}
 \leq  \exp\left(-\frac{(\log y)^2}{2}\right)\sum_{m=1}^{\infty}\frac{d_k(m)}{m^{1/3}}e^{-m/(2y)}
 & \leq  \exp\left(-\frac{(\log y)^2}{2}\right)\left(e^{1/(2y)}\sum_{a=1}^{\infty}\frac{e^{-a/(2y)}}{a^{1/3}}\right)^{k}\\
 & \ll  \exp\left(-\frac{(\log y)^2}{2}\right) y^k \ll \exp\left(-\frac{(\log y)^2}{4}\right),
\end{align*}
using an argument similar to \eqref{Divisor2}. Therefore, we deduce that the error term in \eqref{eq:expected-error-1} is 
$
\ll x^{-1/6} y^{2/3} (\log x)^{k+2} \ll x^{-1/20}.$

We now wish to remove the $e^{-n/y}$ factor from the main term of (\ref{eq:expected-error-1}),
and in so doing we introduce an error of
\begin{equation}\label{ErrorSmooth}
\sum_{m=1}^{\infty}\frac{d_{z}(m)\mathbb{E}(X(m))(1-e^{-m/y})}{m}.
\end{equation}
We shall use the bound $1-e^{-t}\ll t^{\alpha}$ which is valid for all $t>0$ and  $0<\alpha\leq 1$. Also, by Lemma \ref{ExpectationRandom} we have $\vert\mathbb{E}(X(m))\vert\ll m_{0}^{-1}$, where $m_0$ is the squarefree part of $m$. Choosing $\alpha= 1/\log_2 x$, and writing $m=m_0 m_1^2$ we deduce that this sum is 
$$
\ll y^{-\alpha}\sum_{m=1}^{\infty}\frac{d_k(m)}{m_{0}m^{1-\alpha}}\\
\leq  y^{-\alpha}\sum_{m_0=1}^{\infty}
\frac{d_k(m_{0})}{m_{0}^{2-\alpha}}\sum_{m_1=1}^{\infty}\frac{d_k^{2}(m_{1})}{m_{1}^{2-2\alpha}}
= y^{-\alpha}\zeta(2-\alpha)^k \sum_{n=1}^{\infty}\frac{d_k^{2}(n)}{n^{2-2\alpha}}.
$$
Finally, we use the following bound, which follows from Lemma 3.3 of \cite{La1} 
\[
 \sum_{n=1}^{\infty}\frac{d_k^{2}(n)}{n^{2-2\alpha}} \leq \exp\left((2+o(1)) k\log_2 k \right).
 \]
This shows that the sum in \eqref{ErrorSmooth} is 
$ \ll \exp\left(-\log x/(20\log_2x)\right),$
which completes the proof.
\end{proof}

%%%%%%%%%%%%%%%%%%%%%%%%%%%%%%%%%%%%%%%%%%%%%%%%%%%
\section{The distribution of the random model $L(1,\X)$}\label{Sec4}

\subsection{Main results and proof of Theorem \ref{ExponentialDecay}}

Throughout this section, we shall  focus only on proving the desired results for $\Phi_{\X}(\tau)$, since the proofs for $\Psi_{\X}(\tau)$ require only some minor adjustments. Since the $\X(p)$ are independent, then for any $z\in \mathbb{C}$ we have
$$
\ex\left(L(1,\X)^z\right)= \prod_{p} E_p(z),
$$
where
\begin{equation}\label{expectation}
E_p(z):= \ex\left(\left(1-\frac{\X(p)}{p}\right)^{-z}\right)= \alpha_p \left(1-\frac{1}{p}\right)^{-z}+
\beta_p \left(1+\frac{1}{p}\right)^{-z}+\gamma_p.
\end{equation}
For $z\in \mathbb{C}$ we define
$$\Lo(z):=\log\ex\left(L(1,\X)^z\right)=\sum_p\log E_p(z).$$
Let $\tau$ be a large real number and consider the equation
\begin{equation}\label{SaddlePointEq}
 \Big(\ex\left(L(1,\X)^r\right)(e^{\gamma} \tau)^{-r}\Big)'=0 \Longleftrightarrow \Lo'(r)=\log\tau+\gamma, 
 \end{equation}
where the derivative is taken with respect to the real variable $r$. Then it follows from Proposition \ref{MomentRand} below that $\lim_{r\to\infty} \Lo'(r)=\infty$. Moreover, a simple calculation shows that $E_p''(r)E_p(r)>(E_p'(r))^2$ for all primes $p$, and hence that $\Lo''(r)>0$. Thus, we deduce that equation \eqref{SaddlePointEq} has a unique solution $\kappa=\kappa(\tau)$. Using a careful saddle point analysis we obtain an asymptotic formula 
for $\Phi_{\X}(\tau)$ in terms of the moment $\ex\left(L(1,\X)^{r}\right)$ evaluated at the saddle point $\kappa$.
\begin{thm}\label{SaddlePoint}
Let $\tau$ be large and $\kappa$ denote the unique solution to \eqref{SaddlePointEq}. Then,  we have
\begin{equation}\label{MainSaddle}
\Phi_{\X}(\tau)= \frac{\ex\left(L(1,\X)^{\kappa}\right)(e^{\gamma}\tau)^{-\kappa}}{\kappa \sqrt{2\pi \Lo''(\kappa)}}\left(1+O\left(\sqrt{\frac{\log \kappa}{\kappa}} \right)\right).
\end{equation}
Moreover, for any $0\leq \lambda\leq 1/\kappa$ we have 
\begin{equation}\label{SmallShift}
\Phi_{\X}\big(e^{-\lambda}\tau\big)=\Phi_{\X}(\tau)\big(1+O(\lambda \kappa)\big).
\end{equation}
\end{thm}
In order to deduce Theorem \ref{ExponentialDecay} from this result, we need to estimate $\Lo(r)$ and its first few derivatives when $r$ is large. We prove
\begin{pro}\label{MomentRand}
For any real number $r\geq 4$ we have
\begin{equation}\label{MomentRand1}
\Lo(r)=r\left(\log\log r+\gamma+\frac{C_0-1}{\log r}+O\left(\frac{1}{(\log r)^2}\right)\right),
\end{equation}
and
\begin{equation}\label{MomentRand2}
\Lo'(r)= \log\log r +\gamma+\frac{C_0}{\log r}+ O\left(\frac{1}{(\log r)^2}\right).
\end{equation}
Moreover, for all real numbers $y, t$ such that $|y|\geq 3$ and $|t|\leq |y|$ we have
\begin{equation}\label{MomentRand3}
\Lo''(y)\asymp \frac{1}{|y|\log|y|}, \text{ and } \Lo'''(y+it)\ll \frac{1}{|y|^2\log |y|}.
\end{equation}
\end{pro}
Theorem \ref{ExponentialDecay} now follows upon combining Theorem \ref{SaddlePoint} and Proposition \ref{MomentRand}.
\begin{proof}[Proof of Theorem \ref{ExponentialDecay}]
 By Theorem \ref{SaddlePoint} and equation \eqref{MomentRand3}, we have 
\begin{align*}
\Phi_{\X}(\tau)&= \frac{\ex\left(L(1, \X)^{\kappa}\right)(e^{\gamma}\tau)^{-\kappa}}{\kappa \sqrt{2\pi \Lo''(\kappa)}}\left(1+O\left(\sqrt{\frac{\log \kappa}{ \kappa}} \right)\right)\\
&= \exp\Big(\Lo(\kappa)-\kappa(\log \tau+ \gamma)+O(\log \kappa)\Big),
\end{align*}
where $\kappa$ is the unique solution to $\Lo'(\kappa)=\log \tau+\gamma$. Furthermore, by  \eqref{MomentRand2} we have
\begin{equation}\label{EstSaddle1}
\log \tau=\log\log \kappa+\frac{C_0}{\log \kappa}+ O\left(\frac{1}{(\log \kappa)^2}\right),
\end{equation}
and hence we deduce from \eqref{MomentRand1} that 
$$
\Phi_{\X}(\tau)= \exp\left(-\frac{\kappa}{\log \kappa} +O\left(\frac{\kappa}{(\log \kappa)^2}\right)\right).
$$
The estimate \eqref{AsympLargeDeviations} follows upon noting that $\log\kappa=\tau-C_0+O(1/ \tau)$ by \eqref{EstSaddle1}. Finally, using this fact together with \eqref{SmallShift} imply \eqref{Perturbations}. 

\end{proof}

The remaining of this section will be devoted to the proofs of Theorem \ref{SaddlePoint} and Proposition \ref{MomentRand}. We begin by proving the latter.\subsection{Proof of Proposition \ref{MomentRand}} 

We first need some preliminary lemmas. 
\begin{lem}\label{estimate1}  Let $r\geq 4$ be a real number. Then we have 

\begin{equation}\label{approxExLarge1}
\log E_p(r)=\begin{cases} 
-r\log(1-1/p)+O(1) & \text{ if } p\leq r^{2/3}\\
\log\cosh\left(\frac{r}{p}\right)+O\left(\frac{r}{p^2}\right) & \text{ if } p>r^{2/3}.
\end{cases}
\end{equation}
and 
\begin{equation}\label{approxExLarge2}
\frac{E'_p(r)}{E_p(r)}=\begin{cases}  -\log(1-1/p)\left(1+O\left(e^{-r^{1/3}}\right)\right)& \text{ if } p\leq r^{2/3}\\
\frac{1}{p}\tanh\left(\frac{r}{p}\right) + O\left(\frac{1}{p^2}+\frac{r}{p^3}\right) & \text{ if } p>r^{2/3}.
\end{cases}
\end{equation}

\end{lem}

\begin{proof}
We start by proving \eqref{approxExLarge1}. 
First, if $p<r^{2/3}$ then 
\begin{equation}\label{localsmallp}
E_p(r)= \alpha_p \left(1-\frac{1}{p}\right)^{-r}
\Big(1+O\big(\exp\left(-r^{1/3}\right)\big)\Big),
\end{equation}
from which the desired estimate follows in this case. 

Now if $p>r^{2/3}$, we use that $\alpha_p-\beta_p\ll 1/p$ and $\gamma_p\ll 1/p$, together with the bounds $\cosh(t)-1 \ll  t\cosh(t)$ and $\sinh(t)\ll t\cosh(t)$, which are valid for all $t\geq 0$. Thus we derive
\begin{equation}\label{locallargep}
\begin{aligned}
E_p(r)
&= \left(\alpha_p e^{r/p}+\beta_p e^{-r/p}\right)\left(1+O\left(\frac{r}{p^2}\right)\right)+\gamma_p \\
&= (\alpha_p+\beta_p)\cosh\left(\frac{r}{p}\right) \left(1+O\left(\frac{r}{p^2}\right)\right)+\gamma_p\\
&= \cosh\left(\frac{r}{p}\right) \left(1+O\left(\frac{r}{p^2}\right)\right).
\end{aligned}
\end{equation}
which completes the proof of \eqref{approxExLarge1}.

Next, by \eqref{expectation} we have 
$$ E_p'(r)= 
-\alpha_p  \left(1-\frac{1}{p}\right)^{-r} \log \left(1-\frac{1}{p}\right)- \beta_p  \left(1+\frac{1}{p}\right)^{-r}\log \left(1+\frac{1}{p}\right).
$$
For $p<r^{2/3}$ the desired estimate for $E'_p(r)/E_p(r)$ follows from \eqref{localsmallp}. On the other hand, if $p>r^{2/3}$ then  
\begin{align*} 
E_p'(r)
&= \left(\frac{\alpha_p}{p} e^{r/p}-\frac{\beta_p}{p} e^{-r/p}\right)\left(1+O\left(\frac{1}{p}+\frac{r}{p^2}\right)\right)\\
&= \frac{1}{p}\sinh\left(\frac{r}{p}\right)\left(1+O\left(\frac{1}{p}+\frac{r}{p^2}\right)\right)+O\left(\frac{1}{p^2} \cosh\left(\frac{r}{p}\right)\right),
\end{align*}
since both $\alpha_p$ and $\beta_p$ equal $1/2+O(1/p)$. Combining this estimate with \eqref{locallargep} completes the proof.
\end{proof}
Define 
$$ f(t):= \begin{cases} \log \cosh(t) & \text{ if } 0\leq t <1 \\
 \log \cosh(t)- t  & \text{ if } t \geq 1.\end{cases}
$$
Then we have the following standard estimates for $f$ and $f'$.
\begin{lem}[Lemma 4.5 of \cite{La2}]\label{logcosh}  $f$ is bounded on $[0,\infty)$ and $f(t)=t^2/2+O(t^4)$ if $0\leq t <1.$ Moreover we have 
$$
 f'(t)=\begin{cases}  t +O(t^2) & \text{ if } 0< t<1 \\
  O(e^{-2t}) & \text{ if } t > 1.\end{cases}
$$
\end{lem}

We are now ready to prove Proposition \ref{MomentRand}. 
\begin{proof}[Proof of Proposition \ref{MomentRand}]
We only prove \eqref{MomentRand1} and \eqref{MomentRand2} since \eqref{MomentRand3}  follows along the same lines.
First, by Lemmas \ref{estimate1} and \ref{logcosh} we have
\begin{equation}\label{SumPrimesEstimate}
\begin{aligned}
\Lo(r)
&=-r\sum_{p\leq r^{2/3}}\log\left(1-\frac{1}{p}\right)+\sum_{p>r^{2/3}}\log\cosh\left(\frac{r}{p}\right)+ O\left(r^{2/3}\right)\\
&=-r\sum_{p\leq r}\log\left(1-\frac{1}{p}\right)+\sum_{r^{2/3}<p<r^{4/3}} f\left(\frac{r}{p}\right)+ O\left(r^{2/3}\right).
\end{aligned}
\end{equation}
Now, using the prime number theorem in the form $\pi(t)-\text{Li}(t)\ll t/(\log t)^3$, together with partial summation and Lemma \ref{logcosh}, we obtain
\begin{equation}\label{PartialSummation1}
\begin{aligned}
\sum_{r^{2/3}<p<r^{4/3}}f\left(\frac{r}{p}\right) &=\int_{r^{2/3}}^{r^{4/3}} f\left(\frac{r}{t}\right)\frac{dt}{\log t}+O\left(\frac{r}{(\log r)^2}\right)\\
&=\frac{r}{\log r}\int_{r^{-1/3}}^{r^{1/3}}\frac{f(u)}{u^2}du+O\left(\frac{r}{(\log r)^2}\right),
\end{aligned}
\end{equation}
since $\int_0^{\infty}(f(u)(\log u)/u^2)du<\infty$. Extending the integral in the right hand side of this estimate gives
\begin{equation}\label{PartialSummation2}
 \int_{r^{-1/3}}^{r^{1/3}}\frac{f(u)}{u^2}du= \int_0^{\infty} \frac{f(u)}{u^2}du +O\left(r^{-1/3}\right).
\end{equation}
by Lemma \ref{logcosh}. Finally, by a easy integration by parts along with Lemma \ref{logcosh} we have
\begin{align*}
\int_0^{\infty}\frac{f(u)}{u^2}du
&= \int_0^{\infty} \frac{f'(u)}{u}du-\left(\lim_{x\to 1^-}\frac{f(x)}{x}- \lim_{x\to 0^+}\frac{f(x)}{x}+ \lim_{x\to\infty} \frac{f(x)}{x}- \lim_{x\to 1^+}\frac{f(x)}{x}\right)\\
&= \int_0^{\infty} \frac{f'(u)}{u}du-1.
\end{align*}
Collecting the above estimates yields \eqref{MomentRand1}.

Next, we prove \eqref{MomentRand2}. First, similarly to \eqref{SumPrimesEstimate}, we derive from equation \eqref{approxExLarge2} and Lemma \ref{logcosh} that
\begin{align*}
\Lo'(r)&
=  -\sum_{p\leq r^{2/3}}\log\left(1-\frac{1}{p}\right)+\sum_{r^{2/3}<p} \frac{1}{p}\tanh\left(\frac{r}{p}\right)+O\left(r^{-1/3}\right)\\
&=  -\sum_{p\leq r}\log\left(1-\frac{1}{p}\right)+\sum_{r^{2/3}<p<r^{4/3}} \frac{1}{p}f'\left(\frac{r}{p}\right)+O\left(r^{-1/3}\right).
\end{align*}
Finally, using the prime number theorem and partial summation as in \eqref{PartialSummation1} and \eqref{PartialSummation2}, one can see that 
$$
\sum_{r^{2/3}<p<r^{4/3}} \frac{1}{p}f'\left(\frac{r}{p}\right)=\frac{1}{\log r}\int_0^{\infty}\frac{f'(u)}{u}du+O\left(\frac{1}{(\log r)^2}\right),
$$
from which the estimate \eqref{MomentRand2} follows.

\end{proof}

\subsection{Proof of Theorem \ref{SaddlePoint}}
One of the key ingredients in the proof of Theorem \ref{SaddlePoint} is to show that $\ex\left(L(1,\X)^{r+it}\right)/\ex\left(L(1,\X)^{r}\right)$ is rapidly decreasing in $t$ in the range $|t|\geq \sqrt{r\log r}$. To this end, we establish the following lemma, which is the analogue of Lemma 3.2 of \cite{GrSo2}.

\begin{lem}\label{DecayLocal}
Let $r$ be large. If $p>r/4$, then for some positive constant $b_1$ we have 
$$ \frac{|E_p(r+it)|}{E_p(r)}\leq 
\exp\left(-b_1\left(1-\cos \left(t \log\left(\frac{p+1}{p-1}\right)\right)\right)\right).
$$
\end{lem}
\begin{proof}
Let $x_1, x_2, x_3$ be positive real number numbers, 
and $\theta_2, \theta_3$ be real numbers. We shall use the following inequality which is established in the proof of Lemma 3.2 of \cite{GrSo2}:
$$\left| x_1+x_2 e^{i\theta_2}+ x_3 e^{i\theta_3}\right|\leq 
(x_1+x_2+x_3) \exp\left(-\frac{x_1x_3(1-\cos\theta_3)}{(x_1+x_2+x_3)^2}\right).
$$  
Indeed, applying this inequality with $x_1= \alpha_p(1-1/p)^{-r}, x_2= \gamma_p, 
x_3= \beta_p(1+1/p)^{-r}$, and $\theta_2= t\log (1-1/p)$, and $\theta_3= t\log\left(\frac{p-1}{p+1}\right)$ yields the desired bound, since $p>r/4$. 
\end{proof}
Using this lemma, we deduce the following result.
\begin{lem}\label{DecayMoments} Let $r$ be large. Then, there exists a constant $b_2>0$ such that 
$$\frac{\left|\ex\left(L(1,\X)^{r+it}\right)\right|}{\ex\left(L(1,\X)^r\right)}\ll  \begin{cases}\exp\left(-b_2 \frac{t^2}{r\log r} \right) & \text{ if } |t|\leq r/4 \\
 \exp\left(-b_2\frac{|t|}{\log|t|} \right) & \text{ if } |t|> r/4.
\end{cases}
$$

\end{lem}

\begin{proof}
Let $z=r+it$. Since $|E_p(z)|\leq E_p(r)$ we obtain
that for any real numbers $2\leq y_1<y_2$ 
\begin{equation}\label{decay1}
\frac{\left|\ex\left(L(1,\X)^z\right)\right|}{\ex\left(L(1,\X)^r\right)}\leq \prod_{y_1\leq p\leq y_2}\frac{|E_p(z)|}{E_p(r)}.
\end{equation}
Moreover, note that $|t| \log \left(\frac{p+1}{p-1}\right)\sim 2|t|/p$, whence for $|t|\leq p/4$ we have
$$
1-\cos \left(t \log\left(\frac{p+1}{p-1}\right)\right)\gg \frac{|t|^2}{p^2}.
$$ 
If $|t|\leq r/4$ we choose $y_1= r$ and $y_2=2r$.  Then, appealing to Lemma \ref{DecayLocal} gives the desired bound in this case. Finally, in the case $|t|>r/4$, we use the same argument with $y_1= 4|t|$ and $y_2=8|t|$. 
\end{proof}
Let $\varphi(y)=1$ if $y>1$ and equals $0$ otherwise. 
To relate the distribution function of $L(1, \X)$ (or that of $L(1, \chi_d)$ over $d\in \D$) to its complex moments, we use the following smooth analogue of Perron's formula.

\begin{lem}\label{SmoothPerron}
Let $\lambda>0$ be a real number and $N$ be a positive integer. For any $c>0$ we have for $y>0$
\begin{equation}\label{SmoothPerron1}
0\leq \frac{1}{2\pi i}\int_{c-i\infty}^{c+i\infty} y^s \left(\frac{e^{\lambda s}-1}{\lambda s}\right)^N \frac{ds}{s} -\varphi(y)\leq 
\frac{1}{2\pi i}\int_{c-i\infty}^{c+i\infty} y^s \left(\frac{e^{\lambda s}-1}{\lambda s}\right)^N \frac{1-e^{-\lambda N s}}{s}ds,
\end{equation}
and 
\begin{equation}\label{SmoothPerron2}
0\leq \varphi(e^{\lambda}y)-\varphi(y)\leq \frac{1}{2\pi i}\int_{c-i\infty}^{c+i\infty} y^s \left(\frac{e^{\lambda s}-1}{\lambda s}\right) \frac{e^{\lambda s}-e^{-\lambda s}}{s}ds.
\end{equation}
\end{lem}

\begin{proof}
For any $y>0$ we have 
$$\frac{1}{2\pi i}\int_{c-i\infty}^{c+i\infty} y^s \left(\frac{e^{\lambda s}-1}{\lambda s}\right)^N\frac{ds}{s}
= \frac{1}{\lambda^N}\int_{0}^{\lambda}\cdots \int_0^{\lambda} \frac{1}{2\pi i} \int_{c-i\infty}^{c+i\infty}
\left(ye^{t_1+ \cdots+ t_N}\right)^s\frac{ds}{s} dt_1\cdots dt_N
$$
so that by Perron's formula we obtain
\begin{equation}\label{SmoothPerron3}
\frac{1}{2\pi i}\int_{c-i\infty}^{c+i\infty} y^s \left(\frac{e^{\lambda s}-1}{\lambda s}\right)^N\frac{ds}{s}
= \begin{cases} = 1 & \text{ if } y\geq 1, \\ \in [0,1] & \text{ if } e^{-\lambda N } \leq y < 1,\\
 =0 & \text{ if } 0<y< e^{-\lambda N }. \end{cases}
\end{equation}
Therefore we deduce that 
$$
 \frac{1}{2\pi i}\int_{c-i\infty}^{c+i\infty} y^s e^{-\lambda N s} \left(\frac{e^{\lambda s}-1}{\lambda s}\right)^N \frac{ds}{s} \leq \varphi(y)\leq \frac{1}{2\pi i}\int_{c-i\infty}^{c+i\infty} y^s \left(\frac{e^{\lambda s}-1}{\lambda s}\right)^N \frac{ds}{s},
$$
which implies \eqref{SmoothPerron1}. Using these bounds for  $\varphi(y)$ and $\varphi(e^{\lambda}y)$ with $N=1$ gives \eqref{SmoothPerron2}.

\end{proof}

\begin{proof}[Proof of Theorem \ref{SaddlePoint}]

We start by proving \eqref{MainSaddle}. Let $0<\lambda<1/(2\kappa)$ be a real number to be chosen later. Using  \eqref{SmoothPerron1} with $N=1$ we obtain
\begin{equation}\label{approximation1}
\begin{aligned}
0&\leq \frac{1}{2\pi i}\int_{\kappa-i\infty}^{\kappa+i\infty}\ex\left(L(1,\X)^s\right) (e^{\gamma}\tau)^{-s}\frac{e^{\lambda s}-1}{\lambda s}\frac{ds}{s}-\Phi_{\X}(\tau)\\
&\leq \frac{1}{2\pi i}\int_{\kappa-i\infty}^{\kappa+i\infty} \ex\left(L(1,\X)^s\right)(e^{\gamma}\tau)^{-s} \frac{\left(e^{\lambda s}-1\right)}{\lambda s} \frac{\left(1-e^{-\lambda s}\right)}{s}ds.
\end{aligned}
\end{equation}
Since $\lambda\kappa<1/2$ we have
$|e^{\lambda s}-1|\leq 3 \text{ and } |e^{-\lambda s}-1|\leq 2$. 
Hence, using Lemma \ref{DecayMoments} together with the fact that $|\ex\left(L(1, \X)^s\right)|\leq \ex\left(L(1, \X)^{\kappa}\right)$, we obtain for some constant $b_3>0$
\begin{equation}\label{error12}
 \int_{\kappa-i\infty}^{\kappa-i\kappa^{3/5}}+ \int_{\kappa+i\kappa^{3/5}}^{\kappa+i\infty} \ex\left(L(1,\X)^s\right)(e^{\gamma}\tau)^{-s} \frac{e^{\lambda s}-1}{\lambda s}\frac{ds}{s} \ll \frac{e^{-b_3\kappa^{1/6}}}{\lambda \kappa^{3/5}}  \ex\left(L(1,\X)^{\kappa}\right)(e^{\gamma}\tau)^{-\kappa} ,
\end{equation}
and similarly
\begin{equation}\label{error2}
\begin{aligned}
\int_{\kappa-i\infty}^{\kappa-i\kappa^{3/5}}+ \int_{\kappa+i\kappa^{3/5}}^{\kappa+i\infty} &\ex\left(L(1,\X)^s\right)(e^{\gamma}\tau)^{-s}  \frac{\left(e^{\lambda s}-1\right)}{\lambda s} \frac{\left(1-e^{-\lambda s}\right)}{s}ds\\
& \ll \frac{e^{-b_3\kappa^{1/6}}}{\lambda \kappa^{3/5}}  \ex\left(L(1,\X)^{\kappa}\right)(e^{\gamma}\tau)^{-\kappa}.
\end{aligned}
\end{equation}
Let $s=\kappa+it$. If $|t|\leq \kappa^{3/5}$ then $\left|(1-e^{-\lambda s})(e^{\lambda s}-1)\right|\ll \lambda^2|s|^2$, and hence we get
$$
\int_{\kappa-i\kappa^{3/5}}^{\kappa+i\kappa^{3/5}} \ex\left(L(1,\X)^{s}\right)(e^{\gamma}\tau)^{-s}  \frac{\left(e^{\lambda s}-1\right)}{\lambda s} \frac{\left(1-e^{-\lambda s}\right)}{s}ds \ll \lambda\kappa^{3/5} \cdot \ex\left(L(1,\X)^{\kappa}\right)(e^{\gamma}\tau)^{-\kappa} .
$$ 
Therefore, combining this estimate with equations \eqref{approximation1}, \eqref{error12} and \eqref{error2} we deduce that
\begin{equation}\label{approximation2}
\begin{aligned}
\Phi_{\X}(\tau) - &\frac{1}{2\pi i}\int_{\kappa-i\kappa^{3/5}}^{\kappa+i\kappa^{3/5}}\ex\left(L(1,\X)^{s}\right)(e^{\gamma}\tau)^{-s} \frac{e^{\lambda s}-1}{\lambda s^2} ds\\
& \ll \left(\lambda\kappa^{3/5}+\frac{e^{-b_3\kappa^{1/6}}}{\lambda \kappa^{3/5}}\right)\ex\left(L(1,\X)^{\kappa}\right)(e^{\gamma}\tau)^{-\kappa}.
\end{aligned}
\end{equation}
On the other hand, it follows from equation \eqref{MomentRand3} that for $|t|\leq \kappa^{3/5}$ we have 
$$ 
\Lo(\kappa+it)= \Lo(\kappa)+it \Lo'(\kappa)-\frac{t^2}{2}\Lo''(\kappa)+ O\left(\frac{|t|^3}{\kappa^2\log \kappa}\right).
$$
Also, note that
$$ \frac{e^{\lambda s}-1}{\lambda s^2}=\frac{1}{s}\big(1+O(\lambda \kappa)\big)= \frac{1}{\kappa}\left(1-i\frac{t}{\kappa}+ O\left(\lambda \kappa+\frac{t^2}{\kappa^2}\right)\right).$$
Hence, using that $\ex\left(L(1,\X)^{s}\right) =\exp(\Lo(s))$ and $\Lo'(\kappa)=\log \tau+\gamma$ we obtain 
\begin{align*}
&\ex\left(L(1,\X)^{s}\right)(e^{\gamma}\tau)^{-s} \frac{e^{\lambda s}-1}{\lambda s^2}\\
= &\frac{1}{\kappa}\ex\left(L(1,\X)^{\kappa}\right)(e^{\gamma}\tau)^{-\kappa} \exp\left(-\frac{t^2}{2}\Lo''(\kappa)\right) 
\left(1-i\frac{t}{\kappa}+O\left(\lambda\kappa+ \frac{t^2}{\kappa^2}+ \frac{|t|^3}{\kappa^2\log \kappa}\right)\right).\\
\end{align*}
Thus, we get
\begin{equation}\label{TaylorSaddle}
\begin{aligned}
&\frac{1}{2\pi i}\int_{\kappa-i\kappa^{3/5}}^{\kappa+i\kappa^{3/5}}\ex\left(L(1,\X)^{s}\right)(e^{\gamma}\tau)^{-s} \frac{e^{\lambda s}-1}{\lambda s^2} ds\\
=& \frac{1}{\kappa}\ex\left(L(1,\X)^{\kappa}\right)(e^{\gamma}\tau)^{-\kappa} \frac{1}{2\pi} \int_{-\kappa^{3/5}}^{\kappa^{3/5}}\exp\left(-\frac{t^2}{2}\Lo''(\kappa)\right)
\left(1+ O\left(\lambda\kappa+ \frac{t^2}{\kappa^2}+ \frac{|t|^3}{\kappa^2\log \kappa}\right)\right)dt
\end{aligned}
\end{equation}
since the integral involving $it/{\kappa}$ vanishes. Further, since $\Lo''(\kappa)\asymp 1/(\kappa\log \kappa)$ by \eqref{MomentRand3}, then we have for some constant $b_4>0$
$$ 
\frac{1}{2\pi} \int_{-\kappa^{3/5}}^{\kappa^{3/5}}\exp\left(-\frac{t^2}{2}\Lo''(\kappa)\right)dt= \frac{1}{\sqrt{2\pi \Lo''(\kappa)}}\left(1+O\left(e^{-b_4\kappa^{1/6}}\right)\right),
$$
and 
$$ \int_{-\kappa^{3/5}}^{\kappa^{3/5}}|t|^n\exp\left(-\frac{t^2}{2}\Lo''(\kappa)\right)dt\leq  \int_{-\infty}^{\infty}|t|^n\exp\left(-\frac{t^2}{2}\Lo''(\kappa)\right)dt\ll  \frac{1}{\Lo''(\kappa)^{(n+1)/2}}\ll \frac{(\kappa \log \kappa)^{n/2}}{\sqrt{\Lo''(\kappa)}}.$$
Inserting these estimates in \eqref{TaylorSaddle} we deduce that
\begin{equation}\label{main}
\begin{aligned}
&\frac{1}{2\pi i}\int_{\kappa-i\kappa^{3/5}}^{\kappa+i\kappa^{3/5}}\ex\left(L(1,\X)^{s}\right)(e^{\gamma}\tau)^{-s} \frac{e^{\lambda s}-1}{\lambda s^2} ds\\
=& \frac{\ex\left(L(1,\X)^{\kappa}\right)(e^{\gamma}\tau)^{-\kappa}}{\kappa\sqrt{2\pi \Lo''(\kappa)}}
\left(1+ O\left(\lambda\kappa+ \sqrt{\frac{\log \kappa}{\kappa}}\right)\right).
\end{aligned}
\end{equation}
Finally,  combining the estimates \eqref{approximation2} and \eqref{main} and choosing $\lambda= \kappa^{-2}$ completes the proof of \eqref{MainSaddle}.

We now prove \eqref{SmallShift}. Let $0\leq \lambda\leq 1/\kappa$ be a real number. Then, by \eqref{SmoothPerron2} we have
$$ \Phi_{\X}(e^{-\lambda} \tau)-\Phi_{\X}(\tau) \leq 
\frac{1}{2\pi i}\int_{\kappa-i\infty}^{\kappa+i\infty} \ex\left(L(1,\X)^s\right)(e^{\gamma}\tau)^{-s} \frac{\left(e^{\lambda s}-1\right)}{\lambda s} \frac{\left(e^{\lambda s}-e^{-\lambda s}\right)}{s}ds.
$$
We write $s=\kappa+it$, and split the above integral into two parts $|t|\leq \sqrt{\kappa\log\kappa}$ and $|t|> \sqrt{\kappa\log\kappa}$.

Note that both $|(e^{\lambda s}-1)/\lambda s|$ and $|(e^{\lambda s}-e^{-\lambda s})/\lambda s|$ are always less than $4$, which is easily seen by looking at the cases $|\lambda s|\leq 1$ and $|\lambda s|>1$ separately. Therefore, it follows that the contribution of the first part is $\ll \lambda \sqrt{\kappa\log\kappa} \cdot \ex\left(L(1,\X)^{\kappa}\right)(e^{\gamma}\tau)^{-\kappa}$. Furthermore, by Lemma \ref{DecayMoments} we obtain that the contribution of the second part is 
\begin{align*}
&\ll \lambda \ex\left(L(1,\X)^{\kappa}\right)(e^{\gamma}\tau)^{-\kappa}\left(\int_{\sqrt{\kappa\log\kappa}< |t|\leq \kappa/4} e^{-b_2 t^2/(\kappa\log\kappa)} dt+ \int_{|t|\geq \kappa/4} e^{-b_2 |t|/\log |t|} dt\right)\\
& \ll \lambda \sqrt{\kappa\log\kappa} \cdot \ex\left(L(1,\X)^{\kappa}\right)(e^{\gamma}\tau)^{-\kappa}.
\end{align*}
The desired bound follows from \eqref{MainSaddle} and \eqref{MomentRand3}, which show that
\begin{equation}\label{order}
 \Phi_{\X}(\tau)\asymp\frac{\ex\left(L(1, \X)^{\kappa}\right)(e^{\gamma}\tau)^{-\kappa}}{\kappa\sqrt{\Lo''(\kappa)}}
\asymp\sqrt{\frac{\log \kappa}{\kappa}}\cdot \ex\left(L(1, \X)^{\kappa}\right)(e^{\gamma}\tau)^{-\kappa}.
\end{equation}

\end{proof}

%%%%%%%%%%%%%%%%%%%%%%%%%%%%%%%%%%%%%%%%%%%%%%%%%%
\section{The distribution of values of $L(1,\chi_d)$ over $d\in \D$: Proof of Theorems \ref{MainResult} and \ref{TheoremDistribution}}\label{Sec5}

We shall first prove Theorem \ref{TheoremDistribution} and then deduce Theorem \ref{MainResult}. To shorten our notation we let
$$\pr_x(L(1,\chi_d) \in S):=\frac{1}{|\D|}\big|\{d\in \D: L(1,\chi_d) \in S \}\big|,$$
and 
$$M_x(z):= \frac{1}{|\D|}\sumst_{d\in \D} L(1,\chi_d)^z,$$
where as before $\sumst$ indicates that the sum is over non-exceptional discriminants $d$.

\begin{proof}[Proof of Theorem \ref{TheoremDistribution}]
As in Section \ref{Sec4}, let $\kappa$ be the unique solution to $\Lo'(r)=\log\tau+\gamma$. Let $N$ be a positive integer and  $0<\lambda<\min\{1/(2\kappa), 1/N\}$ be a real number to be chosen later. 

Let $Y=b\log x/(\log_2 x\log_3 x)$, for some suitably small constant $b>0$. If $x$ is large enough then equation \eqref{EstSaddle1} insures that $\kappa\leq Y$ in our range of $\tau$. Also, note that Theorem \ref{ComplexMoments} holds for all complex numbers $s=\kappa+it$ with $|t|\leq Y$. 
We consider the integrals
$$
J(\tau)= \frac{1}{2\pi i}\int_{\kappa-i\infty}^{\kappa+i\infty}\ex\left(L(1, \X)^s\right)(e^{\gamma}\tau)^{-s}\left(\frac{e^{\lambda s}-1}{\lambda s}\right)^N\frac{ds}{s}
$$
and 
$$ 
J_x(\tau)= \frac{1}{2\pi i}\int_{\kappa-i\infty}^{\kappa+i\infty}M_x(s)(e^{\gamma}\tau)^{-s}\left(\frac{e^{\lambda s}-1}{\lambda s}\right)^N\frac{ds}{s}.
$$
 Then, it follows from Lemma \ref{SmoothPerron} that
\begin{equation}\label{Mellin1}
 \Phi_{\X}(\tau)\leq J(\tau)\leq \Phi_{\X}(e^{-\lambda N} \tau),
\end{equation}
and 
\begin{equation}\label{Mellin2}
 \pr_x\Big(L(1,\chi_d)>e^{\gamma}\tau\Big)+O\left(\frac{\log x}{\sqrt{x}}\right)\leq J_x(\tau)\leq \pr_x\Big(L(1,\chi_d)>e^{\gamma-\lambda N}\tau\Big)+O\left(\frac{\log x}{\sqrt{x}}\right),
\end{equation}
since there are at most $O(\log x)$ exceptional discriminants $d\leq x$.
Now, using that $|e^{\lambda s}-1|\leq 3$ we get
$$
\int_{\kappa-i\infty}^{\kappa-iY}+ \int_{\kappa+iY}^{\kappa+i\infty}\ex\left(L(1, \X)^s\right)(e^{\gamma}\tau)^{-s}\left(\frac{e^{\lambda s}-1}{\lambda s}\right)^N\frac{ds}{s}\ll \frac{1}{N}\left(\frac{3}{\lambda Y}\right)^N\ex\left(L(1, \X)^{\kappa}\right)
(e^{\gamma}\tau)^{-\kappa}.
$$
A similar argument together with Theorem \ref{ComplexMoments}  shows that 
\begin{align*}
\int_{\kappa-i\infty}^{\kappa-iY}+ \int_{\kappa+iY}^{\kappa+i\infty}M_x(s) (e^{\gamma}\tau)^{-s}\left(\frac{e^{\lambda s}-1}{\lambda s}\right)^N\frac{ds}{s}
&\ll \frac{1}{N}\left(\frac{3}{\lambda Y}\right)^NM_x(\kappa)(e^{\gamma}\tau)^{-\kappa}\\
&\ll \frac{1}{N}\left(\frac{3}{\lambda Y}\right)^N\ex\left(L(1, \X)^{\kappa}\right)(e^{\gamma}\tau)^{-\kappa}.
\end{align*}
Combining these bounds with Theorem \ref{ComplexMoments} and using that $|(e^{\lambda s}-1)/\lambda s|\leq 4$ we derive
\begin{equation}\label{difference1}
J_x(\tau)- J(\tau)\ll \frac1N\left(\frac{3}{\lambda Y}\right)^N\ex\left(L(1, \X)^{\kappa}\right)(e^{\gamma}\tau)^{-\kappa}+  \frac{Y}{\kappa}4^N (e^{\gamma}\tau)^{-\kappa} \exp\left(-\frac{\log x}{20\log_2 x}\right).
\end{equation}
Thus, choosing $N=[\log\log x]$ and $\lambda= e^{10}/Y$ we deduce from \eqref{order} that 
\begin{equation}\label{difference2}
J_x(\tau)- J(\tau)\ll \frac{1}{(\log x)^{3}} \Phi_{\X}(\tau).
\end{equation}
On the other hand, it follows from Theorem \ref{ExponentialDecay} that 
$$
\Phi_{\X}(e^{\pm \lambda N}\tau)
= \Phi_{\X}(\tau)\left(1+O\left(\frac{e^{\tau}(\log_2 x)^2\log _3 x}{\log x}\right)\right).
$$
Combining this last estimate with \eqref{Mellin1}, \eqref{Mellin2}, and \eqref{difference2} we obtain
\begin{align*}
\pr_x(L(1,\chi_d)>e^{\gamma}\tau)&
\leq J_x(\tau)+O\left(\frac{\log x}{\sqrt{x}}\right) \\
&\leq J(\tau)+ O\left(\frac{\Phi_{\X}(\tau)}{(\log x)^{5}}+\frac{\log x}{\sqrt{x}}\right)\\
&\leq \Phi_{\X}(\tau)\left(1+O\left(\frac{e^{\tau}(\log_2 x)^2\log _3 x}{\log x}\right)\right)+ O\left(\frac{\log x}{\sqrt{x}}\right),
\end{align*}
and 
\begin{align*}
\pr_x(L(1,\chi_d)>e^{\gamma}\tau)&
\geq J_x(e^{\lambda N}\tau)+O\left(\frac{\log x}{\sqrt{x}}\right) \\
&\geq J(e^{\lambda N}\tau)+ O\left(\frac{\Phi_{\X}(\tau)}{(\log x)^{5}}+ \frac{\log x}{\sqrt{x}}\right)\\
&\geq \Phi_{\X}(\tau)\left(1+O\left(\frac{e^{\tau}(\log_2 x)^2\log _3 x}{\log x}\right)\right)+ O\left(\frac{\log x}{\sqrt{x}}\right).\\
\end{align*}
The result follows from these estimates together with the fact that  $\Phi_{\X}(\tau)\gg x^{-1/4}$ in our range of $\tau$, by Theorem \ref{ExponentialDecay}.

\end{proof}

We now deduce Theorem \ref{MainResult}.

\begin{proof}[Proof of Theorem \ref{MainResult}]
By the class number formula \eqref{Class2}, we have $h(d)\geq  2e^{\gamma}\frac{\sqrt{d}}{\log d} \cdot \tau $ if and only if 
$$L(1,\chi_d)\geq e^{\gamma} \tau\left(1+ 2\frac{\log\left(1+\sqrt{1-1/d}\right)}{\log d}\right).$$
The desired estimate follows from Theorems \ref{TheoremDistribution} and \ref{ExponentialDecay}, which show that the number of $d\in \D$ such that $d\geq \sqrt{x}$ and $L(1,\chi_d)\geq e^{\gamma}\tau (1+O(1/\log d))$ is 
\begin{align*}
 &|\D|\cdot \Phi_{\X}\Big(\tau \big(1+O(1/\log x)\big)\Big)\left(1+O\left(\frac{e^{\tau}(\log_2 x)^2\log _3 x}{\log x}\right)\right)\\
 =&|\D| \cdot  \exp\left(-\frac{e^{\tau-C_0}}{\tau}\left(1+ O\left(\frac{1}{\tau}\right)\right)\right).
 \end{align*}
The analogous estimate for the number of discriminants $d\in \D$ such that $h(d)\leq 2e^{-\gamma}\zeta(2)\frac{\sqrt{d}}{\log d}\cdot \frac{1}{\tau}$ follows along the same lines.
\end{proof}

%%%%%%%%%%%%%%%%%%%%%%%%%%%%%%%%%%%%%%%%%%%%%%%%%%%
\section{The number of quadratic fields with a given class number: Proof of Theorem \ref{thm:F-average}}
\label{Sec6}
Recall that $\F$ is  the number of discriminants in the family $\mathcal{D}_{\text{ch}}$ with class number $h$. In order to obtain an asymptotic formula for $\sum_{h\leq H} \F$,
we first show that we can restrict our attention to discriminants $d\in \mathcal{D}_{\text{ch}}$ such that  $d\leq X:=H^2(\log H)^8$.  To this end we use Tatuzawa's refinement of Siegel's Theorem \cite{Ta}, which states that for large $d$, we have $L(1,\chi_d)\geq 1/(\log d)^2$ with at most one exception. This implies that $h(d)\geq \sqrt{d}\cdot(\log d)^{-3}$ with at most one exception, by the class number formula \eqref{Class2}. Thus, if $h(d)\leq H$ then we must have $d\leq X$, with at most one exception. This yields
  \begin{equation}
\label{eq:F-average-truncation}
	\sum_{h\leq H} \F
	= \sum_{\substack{d \in \mathcal{D}_{\textup{ch}}(X)\\ h(d) \leq H}} 1
	+ O(1).
\end{equation}
\begin{proof}[Proof of Theorem \ref{thm:F-average}] 
We estimate the main term in \eqref{eq:F-average-truncation} by using the smoothing function
\[
	I_{c,\lambda,N}(y):=\frac{1}{2\pi i}\int_{c-i\infty}^{c+i\infty}y^{s}\left(\frac{e^{\lambda s}-1}{\lambda s}\right)^{N}\frac{ds}{s},
\]
where $c=1/\log H$, $N$ is a positive integer, and $0<\lambda\leq 1$ is a real number to be chosen later. Using (\ref{eq:F-average-truncation}) together with  \eqref{SmoothPerron3}, we obtain
\begin{equation}
\label{eq:F(h)-integral-inequality}
\sum_{h\leq H}\F \leq\frac{1}{2\pi i}\int_{c-i\infty}^{c+i\infty}\sum_{\substack{d\in \mathcal{D}_{\textup{ch}}(X)}
}\frac{H^{s}}{h(d)^{s}}\left(\frac{e^{\lambda s}-1}{\lambda s}\right)^{N}\frac{ds}{s}+O\left(1\right)\leq\sum_{h\leq e^{\lambda N}H}\F.
\end{equation}
By  \eqref{MontWein}, Theorem \ref{ComplexMoments}, and Remark \ref{RemExceptional}, there exists a constant $B>0$ such that for all $x\geq \sqrt{X}$ and any complex number $z$ with $\re(z) > -1/2$ and  $\vert z \vert \leq T := B \log X / (\log_2 X \log_3 X) $, we have
\begin{equation}\label{RecallComplex}
\sum_{d\in\mathcal{D}_{\textup{ch}}(x)}L(1,\chi_{d})^{z}=C_{1}x^{1/2}\mathbb{E}(L(1,\mathbb{X})^{z})+O\left(x^{1/2}\exp\left(-\frac{\log x}{20\log\log x}\right)\right)
\end{equation}
where
\[
	C_{1}=\frac{1}{2}\prod_{p>2}\left(1-\frac{c(p)}{p^{2}}\right).
\]
For brevity, we define
$$
	\ell(x):=\frac{\sqrt{x}}{\log(\sqrt{x-1}+\sqrt{x})}.
$$
Then we have $h(d)=\ell(d)L(1, \chi_d)$ by the class number formula \eqref{Class2}. Hence, using integration by parts, we deduce from \eqref{RecallComplex} that
\begin{multline}
\label{eq:average-h(d)^(-s)}
	\sum_{\substack{d\in\mathcal{D}_{\text{ch}}(X)}}
	h(d)^{-s}=\frac{C_{1}}{2}\mathbb{E}(L(1,\mathbb{X})^{-s})
	\left(
		\int_{1}^{X}x^{-1/2}\ell(x)^{-s}		dx
	\right)
	\\
	+O\left(X^{1/2}\exp\left(-\frac{\log X}{50\log\log X}\right)\right)
\end{multline}
for $\vert s \vert \leq T $ and $\text{Re}(s) = c$.

 Since $h(d)\geq 1$ and $\vert e^{\lambda s}-1\vert\leq 3$ for large enough $H$, we see that the contribution of the region $\vert s\vert>T$ to the integral in (\ref{eq:F(h)-integral-inequality}) is 
\[
\ll X^{1/2}\left(\frac{3}{\lambda}\right)^{N}\int_{\substack{\vert s\vert>T\\
\re(s)=c
}
}\frac{\vert ds\vert}{\vert s\vert^{N+1}}\ll\frac{X^{1/2}}{N}\left(\frac{3}{\lambda T}\right)^{N}.
\]
We also have $\vert(e^{\lambda s}-1)/\lambda s\vert\leq 4$ for large enough $H$. Therefore, it follows from (\ref{eq:average-h(d)^(-s)}) that the integral in (\ref{eq:F(h)-integral-inequality}) equals
\begin{equation}
\label{eq:F-truncated-integral}
\frac{1}{2\pi i}\int_{\substack{\vert s\vert\leq T\\
\re(s)=c
}
}\frac{C_{1}}{2}\mathbb{E}(L(1,\mathbb{X})^{-s})\left(\int_{1}^{X}x^{-1/2}\ell(x)^{-s}	dx\right)H^{s}\left(\frac{e^{\lambda s}-1}{\lambda s}\right)^{N}\frac{ds}{s}+\mathcal{E}
\end{equation}
where
\[
	\mathcal{E}\ll\frac{X^{1/2}}{N}\left(\frac{3}{\lambda T}\right)^{N}+\frac{4^{N}T}{c}X^{1/2}\exp\left(-\frac{\log X}{50\log\log X}\right).
\]
Choosing $\lambda=e^{10}/T$ and $N=[A\log\log H]$ for a constant $A>1$ gives
\[
	\mathcal{E}\ll_{A}\frac{H}{(\log H)^{A}}.
\]
Extending the main term of (\ref{eq:F-truncated-integral}) to the entire line $\re(s)=c$, we see that it equals
\begin{multline}
\label{eq:F-expected-integral}
\frac{1}{2\pi i}\int_{c-i\infty}^{c+i\infty}\frac{C_{1}}{2}\mathbb{E}(L(1,\mathbb{X})^{-s})\left(\int_{1}^{X}x^{-1/2}\ell(x)^{-s}dx\right)H^{s}\left(\frac{e^{\lambda s}-1}{\lambda s}\right)^{N}\frac{ds}{s}\\+O\left(\mathbb{E}\big(L(1,\mathbb{X})^{-c}\big)\frac{X^{1/2}}{N}\left(\frac{3}{\lambda T}\right)^{N}\right)
\\
	= \frac{C_1}{2} \mathbb{E} \left(
		\int_1^X I_{c,\lambda,N} \left(
			\frac{H}{\ell(x) L(1,\mathbb{X})}\right)
		x^{-1/2} dx
	\right)
	+ O_A\left( \frac{H}{(\log H)^A} \right).	
\end{multline}
To shorten our notation we define 
$Y = H L(1,\mathbb{X})^{-1}$. Then
it follows from \eqref{SmoothPerron3} that for $1< x\leq X$ we have
\[
I_{c,\lambda,N}\left(\frac{H}{\ell(x) L(1,\mathbb{X})}\right)=\begin{cases}
1 & \text{ if }\ell(x)\leq Y,\\
\in[0,1] & \text{ if }Y< \ell(x)\leq e^{\lambda N}Y,\\
0 & \text{ if }\ell(x)>e^{\lambda N}Y.
\end{cases}
\]
Furthermore, note that  $
\ell(x)=(2\sqrt{x})/(\log x+\psi(x))
$ for some $\psi(x)$ that satisfies $0\leq \psi(x)\leq  2$. Thus, if for a constant $c$ we define
\[
	\ell_c(x)=\frac{2\sqrt{x}}{\log x + c},
\]
then we have $ \ell_{2}(x) \leq \ell(x) \leq \ell_0(x)$, and therefore
\[
I_{c,\lambda,N}\left(\frac{H}{\ell(x) L(1,\mathbb{X})}\right)=\begin{cases}
1 & \text{ if }\ell_{0}(x)\leq Y,\\
0 & \text{ if }\ell_2(x)>e^{\lambda N}Y,\\
\in[0,1] & \text{otherwise}.\\

\end{cases}
\]
For any $c>0$ the function $\ell_c(x)$ is strictly increasing on $(e^2, \infty)$ and hence is invertible on this domain.  Let $g_c$ be its inverse function.   Then we obtain
\begin{multline}
\label{eq:integral-I-dx}
\int_{1}^{X}I_{c,\lambda,N}\left(\frac{H}{\ell(x) L(1,\mathbb{X})}\right)x^{-1/2}dx=2\min\left(g_0(Y)^{1/2},X^{1/2}\right)
\\
+O\left(g_2(e^{\lambda N}Y)^{1/2}-g_0(Y)^{1/2}+1\right).
\end{multline}
Note that for any $c>0$ we have 
	$g_c(x)=x^{2}\big(\log x+ O_c(\log\log x)\big)^2$ for $x\geq e^2$. Moreover, if $g_0(Y)> X$ then $Y> \ell_0(X)$ and hence $L(1,\X)\ll 1/(\log H)^3 $. 
Therefore, it follows from Theorem \ref{ExponentialDecay} that
\begin{align*}
\mathbb{E}\left(\min\left(g_0(Y)^{1/2},X^{1/2} \right)\right)
&=\mathbb{E}\left(g_0(Y)^{1/2}\right) +O\Big(X^{1/2} \exp\left(-\log^2 H\right)\Big)\\
&= \mathbb{E}\left(L(1, \X)^{-1}\right) H\log H  + O(H\log_2 H).
\end{align*}
Furthermore, a similar argument shows that
\begin{align*}
\ex\left(g_2(e^{\lambda N}Y)^{1/2}-g_0(Y)^{1/2}\right)
&= \left(e^{\lambda N}-1\right) \mathbb{E}\left(L(1, \X)^{-1}\right) H\log H + O(H\log_2 H)\\
&\ll H (\log_2 H)^2 \log_3 H.
\end{align*}
Combining these estimates with  with equations (\ref{eq:F(h)-integral-inequality}), (\ref{eq:F-truncated-integral}), (\ref{eq:F-expected-integral}) and (\ref{eq:integral-I-dx}) we deduce
$$
\sum_{h\leq H}\F \leq C_1\mathbb{E}\left(L(1, \X)^{-1}\right)H\log H + O\left(H (\log_2 H)^2 \log_3 H\right)\leq\sum_{h\leq e^{\lambda N}H}\F.
$$
Replacing $e^{\lambda N}H$ by $H$ in the right hand side inequality yields
$$
\sum_{h\leq H}\F= C_1\mathbb{E}\left(L(1, \X)^{-1}\right)H\log H + O\big(H (\log_2 H)^2 \log_3 H\big).
$$
Finally, by \eqref{expectation} and the independence of the $\X(p)$'s, we find that $C_1\ex(L(1,\X)^{-1})$ equals
\begin{align*}
	& \frac12\prod_{p>2}\left(1-\frac{c(p)}{p^2}\right)\left(\alpha_p \left(1-\frac1p\right)+\beta_p\left(1+\frac1p\right)+\gamma_p\right)\\
= & \frac12\prod_{p>2}\left(\frac{1}{2} \left(1-\frac{c(p)+1}{p}\right)\left(1-\frac1p\right)+ \frac{1}{2} \left(1-\frac{c(p)-1}{p}\right)\left(1+\frac1p\right)+ c(p)\left(\frac1p-\frac{1}{p^2}\right)\right)\\
=& \frac12\prod_{p>2} \left(1-\frac{c(p)-1}{p^2}\right),\\
\end{align*}
which completes the proof.
\end{proof}


\begin{thebibliography}{DDDD}

\bibitem[1] {Bi1} A. Bir\'o,
\emph{Yokoi's conjecture.}
 Acta Arith. 106 (2003), no. 1, 85--104. 
 
 \bibitem[2] {Bi2} A. Bir\'o,
\emph{Chowla's conjecture.}
 Acta Arith. 107 (2003), no. 2, 179--194. 
 
 \bibitem[3] {ChFr} S. Chowla and J. Friedlander,
\emph{Class numbers and quadratic residues.} 
 Glasgow Math. J. 17 (1976), 47--52.
 
\bibitem[4]{Da} H. Davenport,  
\emph{Multiplicative number theory}.
Third edition. Revised and with a preface by Hugh L. Montgomery. Graduate Texts in Mathematics, 74. Springer-Verlag, New York, 2000. xiv+177 pp.

\bibitem[5] {GrSo1} A. Granville and K. Soundararajan,
\emph{Large character sums.} 
 J. Amer. Math. Soc. 14 (2001), no. 2, 365--397.

\bibitem[6] {GrSo2} A. Granville and K. Soundararajan,
\emph{The distribution of values of $L(1, \chi_d)$.} 
Geom. Funct. Anal. 13 (2003), no. 5, 992--1028. 

\bibitem[7] {HB} D. R. Heath-Brown, 
\emph{A mean value estimate for real character sums.}
Acta Arith. 72 (1995), no. 3, 235--275.

\bibitem[8] {HJKMP} S. Holmin, N. Jones, P. Kurlberg, C. McLeman, K. L. Petersen, 
\emph{Missing class groups and class number statistics for imaginary quadratic fields.} 
Preprint, 28 pages. 	arXiv:1510.04387.

\bibitem[9]{La1}  Y.  Lamzouri,
\emph{Extreme values of $\arg L(1,\chi)$.}
 Acta Arith. 146 (2011), no. 4, 335--354.
 
 \bibitem[10]{La2}  Y.  Lamzouri,
\emph{The distribution of Euler-Kronecker constants of quadratic fields.}
J. Math. Anal. Appl. 432 (2015), no. 2, 632--653.
 
 \bibitem[11]{La3}  Y.  Lamzouri,
\emph{Extreme Values of Class Numbers of Real Quadratic Fields.}
Int. Math. Res. Not. IMRN (2015), no. 22, 11847--11860.
 
  \bibitem[12]{La4}  Y.  Lamzouri,
\emph{On the average of the number of imaginary quadratic fields with a given class
number.}
Preprint, 5 pages. arXiv:1512.07134. 


\bibitem[13]{Li}  J. E. Littlewood,
\emph{On the class number of the corpus $P(\sqrt{-k})$.} 
Proc. London Math. Soc. 27 (1928), 358--372.

\bibitem[14]{Mo}  R. A. Mollin,
\emph{An overview of the solution to the class number one problem for real quadratic fields of Richaud-Degert type.} 
Number theory, Vol. II (Budapest, 1987), 871--888, 
Colloq. Math. Soc. János Bolyai, 51, North-Holland, Amsterdam, 1990. 

\bibitem[15]{MoWi}  R. A. Mollin and H. C. Williams
\emph{Solution of the class number one problem for real quadratic fields of extended Richaud-Degert type (with one possible exception).} 
Number theory (Banff, AB, 1988), 417--425, de Gruyter, Berlin, 1990. 


\bibitem[16] {MoWe} H. L. Montgomery and  J. P. Weinberger, 
\emph{Real quadratic fields with large class number.} 
Math. Ann. 225 (1977), no. 2, 173--176. 

\bibitem[17]{So}  K. Soundararajan,
\emph{The number of imaginary quadratic fields with a given class number.} 
Hardy-Ramanujan J. 30 (2007), 13--18.

\bibitem[18] {St} T. Storer,
\emph{Cyclotomy and difference sets.} 
Lectures in Advanced Mathematics, No. 2 Markham Publishing Co., Chicago, Ill. 1967 vii+134 pp.

\bibitem[19] {Ta} T. Tatuzawa,
\emph{On a theorem of Siegel. }
Jap. J. Math. 21 (1951), 163--178. 


\bibitem[20] {Wa} M. Watkins, 
\emph{Class numbers of imaginary quadratic fields.} 
Math. Comp. 73 (2004), no. 246, 907--938. 

\bibitem[21] {Yo} H. Yokoi, 
\emph{Class number one problem for certain kind of real quadratic fields.} 
Proc. Internat. Conf. (Katata, 1986), Nagoya Univ., Nagoya, 1986, 125--137.

\end{thebibliography}
\end{document}